\documentclass[11pt]{article}

\usepackage[a4paper,margin=1in]{geometry}
\usepackage[T1]{fontenc}
\usepackage[utf8]{inputenc}
\usepackage{lmodern}
\usepackage{microtype}
\usepackage{graphicx}
\usepackage{epstopdf}
\usepackage{amsmath,amssymb,amsfonts,amsthm,mathtools}
\usepackage{bm}
\usepackage{accents}
\usepackage{algorithm2e}
\usepackage{booktabs}
\usepackage{multirow}
\usepackage{enumitem}
\usepackage{subcaption}
\usepackage{xspace}
\usepackage{url}
\usepackage[numbers]{natbib}
\usepackage[hidelinks]{hyperref}
\usepackage{xcolor}
\newcommand{\email}[1]{\href{mailto:#1}{#1}}

\setlist[itemize]{noitemsep, topsep=2pt}
\setlist[enumerate]{noitemsep, topsep=2pt}
\DontPrintSemicolon
\RestyleAlgo{ruled}
\LinesNumbered

\newtheorem{thm}{Theorem}

\theoremstyle{definition}
\newtheorem{defn}{Definition}
\newtheorem{exa}{Example}
\newtheorem{rem}{Remark}

\newcommand{\cX}{\underline{\bf X}}

\newcommand{\oast}{\mathbin{\odot}}

\title{A New Tensor Network: Tubal Tensor Train and Its Applications}
\author{Salman Ahmadi-Asl\thanks{Lab of Machine Learning and Knowledge Representation, Innopolis University, Innopolis, Russia 
(\email{s.ahmadiasl@innopolis.ru}).}
\and Valentin Leplat\thanks{Innopolis University, Innopolis, Russia, (\email{v.leplat@innopolis.ru}).}  \and Anh-Huy Phan\thanks{Skolkovo Institute of Science and Technology, Center for Artificial Intelligence Technology, Moscow, Russia, (\email{a.phan@skoltech.ru}).} \and Andrzej Cichocki\thanks{Systems Research Institute of Polish Academy of Science, Warsaw, Poland (\email{cichockiand@gmail.com}).\\
The first two authors contributed equally to this work.} 
}
\date{}
\begin{document}
\maketitle

\begin{abstract}
We introduce the tubal tensor train (TTT) decomposition, a tensor-network model that combines the t-product algebra of the tensor singular value decomposition (T-SVD) with the low-order core structure of the tensor train (TT) format. For an order-$(N+1)$ tensor with a distinguished tube mode, the proposed representation consists of two third-order boundary cores and $N-2$ fourth-order interior cores linked through the t-product. As a result, for bounded tubal ranks, the storage scales linearly with the number of modes, in contrast to direct high-order extensions of T-SVD. We present two computational strategies: a sequential fixed-rank construction, called TTT-SVD, and a Fourier-slice alternating scheme based on the alternating two-cores update (ATCU). We also state a TT-SVD-type error bound for TTT-SVD and illustrate the practical performance of the proposed model on image compression, video compression, tensor completion, and hyperspectral imaging.
\end{abstract}
\medskip


\section{Introduction}

Tensor decompositions are now standard tools for representing multidimensional data in a compact and structured form. They play an important role in machine learning, computer vision, signal processing, and scientific computing, where the data naturally appear as images, videos, hyperspectral cubes, or higher-order feature arrays \cite{sidiropoulos2017tensor,kolda2009tensor,papalexakis2016tensors}. Classical models such as the canonical polyadic decomposition (CPD) \cite{hitchcock1927expression,hitchcock1928multiple}, Tucker decomposition \cite{tucker1963implications,de2000multilinear}, block-term decomposition \cite{de2008decompositionsIII}, tensor train (TT) decomposition \cite{oseledets2011tensor}, and tensor chain/ring decompositions \cite{espig2012note,zhao2016tensor} each offer a different trade-off between compression, flexibility, and numerical tractability.

Among these models, the tensor singular value decomposition (T-SVD) has become especially attractive for third-order tensors because it relies on the t-product and inherits many useful features of the matrix SVD \cite{kilmer2011factorization,kilmer2013third}. This makes T-SVD effective for compression, denoising, and completion. A key limitation, however, appears when one moves to higher-order tensors: direct extensions of T-SVD retain cores whose order grows with the order of the data tensor, which leads to a curse-of-dimensionality bottleneck \cite{martin2013order,wang2022hot}. In short, the algebra remains elegant, but the representation becomes progressively less practical.

The objective of this paper is to bridge the gap between the favorable t-product algebra of T-SVD and the scalable low-order-core structure of TT. To this end, we introduce the \emph{tubal tensor train} (TTT) decomposition. The idea is to preserve the distinguished tube mode and its associated convolutional structure, while organizing the remaining modes in a train topology. The resulting model replaces a single high-order t-product factorization by a sequence of third- and fourth-order tubal cores linked by the t-product. For moderate tubal ranks, the storage then scales linearly with the number of modes.

Our work is related to several recent attempts to generalize t-product-based models. The convolutional tensor-train model of \citet{su2020convolutional} replaces the standard TT contractions by convolution-type operations for spatio-temporal learning; our construction is different in that it is explicitly built on the t-product and preserves the distinguished tube mode throughout the network. Higher-order variants of T-SVD were proposed in \citet{wang2022hot}, where a tensor--tensor-product analogue of HOSVD is developed together with truncated and sequentially truncated variants. The relationship between T-SVD and HOSVD was further studied in \citet{zeng2020decompositions}, which also motivated the orientation SVD model and its randomized extension \cite{ding2023randomized}. The present paper follows a different route: instead of redefining a single global high-order t-product factorization, we combine the t-product with a train network in order to retain low-order cores while preserving the tubal algebra.

\paragraph{Contributions.}
The main contributions are:
\begin{itemize}[leftmargin=*]
    \item We introduce the tubal tensor train (TTT) decomposition, a new tensor-network model that combines the t-product structure of T-SVD with the train topology of TT.
    \item We show how TTT avoids the high-order-core bottleneck of direct T-SVD extensions by using only third-order and fourth-order tubal cores.
    \item We present two practical algorithms: a fixed-rank sequential construction (TTT-SVD) and a Fourier-domain alternating scheme based on ATCU (TATCU).
    \item We report numerical experiments on RGB images, videos, tensor completion, and hyperspectral images.
\end{itemize}

The paper is organized as follows. Section~\ref{Sec:prelim} introduces the notation and recalls the t-product framework. Section~\ref{Sec:tSVD} summarizes the TT and T-SVD models that motivate our construction. Section~\ref{Sec_TTT} defines the TTT format and presents the proposed algorithms together with an error bound for TTT-SVD. Numerical experiments are reported in Section~\ref{Sec:Sim}. Extended hyperspectral benchmarks and supplementary visualizations, including additional graphical explanations of TT, T-SVD, and the TTT-SVD procedure, are collected in the appendices; see in particular Appendix~\ref{app:more_tt_ttt_algo}.

\section{Preliminaries on the t-product algebra}\label{Sec:prelim}

We work with real-valued tensors throughout the paper; the complex-valued case follows analogously. Since the proposed model is built on the t-product framework, we begin with the third-order tensor setting where the t-product and T-SVD are naturally defined. We then introduce the hyper-vector and hyper-matrix viewpoint that allows these ideas to be used inside a train network for higher-order tensors.

\subsection{Third-order tensors, slices, fibers, tubes, and the t-product}

Let
\[
\underline{\mathbf X}\in\mathbb{R}^{I_1\times I_2\times I_3}
\]
be a third-order tensor. Its frontal, lateral, and horizontal slices are
\[
\underline{\mathbf X}(:,:,i),\qquad
\underline{\mathbf X}(:,i,:),\qquad
\underline{\mathbf X}(i,:,:),
\]
respectively. Fibers are obtained by fixing all but one index; in particular,
\[
\underline{\mathbf X}(i,j,:)
\]
is called a \emph{tube}. In the t-product framework, the third mode is distinguished: it is the mode along which circular convolution is performed. For compactness, we write
\[
\underline{\mathbf X}^{(i)}:=\underline{\mathbf X}(:,:,i),\qquad i=1,\dots,I_3.
\]

A third-order tensor can also be viewed as a \emph{hyper-matrix}, that is, a matrix whose entries are tubes of length $I_3$. This viewpoint is useful because the t-product behaves like matrix multiplication in which scalar multiplication is replaced by circular convolution of tubes.

\begin{defn}[t-product {\cite{kilmer2011factorization}}]
Let
\[
\underline{\mathbf X}\in\mathbb{R}^{I_1\times I_2\times I_3},
\qquad
\underline{\mathbf Y}\in\mathbb{R}^{I_2\times I_4\times I_3}.
\]
Their t-product
\[
\underline{\mathbf C}=\underline{\mathbf X}*\underline{\mathbf Y}
\in\mathbb{R}^{I_1\times I_4\times I_3}
\]
is defined by
\begin{equation}\label{eq:tprod-bcirc}
\underline{\mathbf C}
=
\mathrm{fold}\!\left(
\mathrm{circ}(\underline{\mathbf X})\,
\mathrm{unfold}(\underline{\mathbf Y})
\right),
\end{equation}
where
\[
\mathrm{circ}(\underline{\mathbf X})
=
\begin{bmatrix}
\underline{\mathbf X}^{(1)} & \underline{\mathbf X}^{(I_3)} & \cdots & \underline{\mathbf X}^{(2)}\\
\underline{\mathbf X}^{(2)} & \underline{\mathbf X}^{(1)}   & \cdots & \underline{\mathbf X}^{(3)}\\
\vdots & \vdots & \ddots & \vdots\\
\underline{\mathbf X}^{(I_3)} & \underline{\mathbf X}^{(I_3-1)} & \cdots & \underline{\mathbf X}^{(1)}
\end{bmatrix},
\qquad
\mathrm{unfold}(\underline{\mathbf Y})
=
\begin{bmatrix}
\underline{\mathbf Y}^{(1)}\\
\underline{\mathbf Y}^{(2)}\\
\vdots\\
\underline{\mathbf Y}^{(I_3)}
\end{bmatrix},
\]
and $\mathrm{fold}$ is the inverse operation of $\mathrm{unfold}$.
\end{defn}

Equivalently, the t-product may be described entrywise in terms of circular convolution of tubes. If the $(i,j)$ tube of $\underline{\mathbf X}$ is convolved with the $(j,k)$ tube of $\underline{\mathbf Y}$ and summed over $j$, one obtains the $(i,k)$ tube of $\underline{\mathbf C}$. Hence, the t-product is a tensor analogue of matrix multiplication in which scalar multiplication is replaced by circular convolution of tubes. For two tubes $\mathbf x,\mathbf y\in\mathbb{R}^{I_3}$, their circular convolution is denoted by $\mathbf x\circledast \mathbf y$.

\subsection{Fourier-domain interpretation}

The t-product is computationally attractive because block-circulant matrices are block diagonalized by the discrete Fourier transform. Let
\[
\widehat{\underline{\mathbf X}}=\mathrm{fft}(\underline{\mathbf X},[],3),
\qquad
\widehat{\underline{\mathbf Y}}=\mathrm{fft}(\underline{\mathbf Y},[],3),
\]
and let $\mathbf F_{I_3}\in\mathbb{C}^{I_3\times I_3}$ denote the discrete Fourier transform matrix. Then
\[
(\mathbf F_{I_3}\otimes \mathbf I_{I_1})\,\mathrm{circ}(\underline{\mathbf X})\,(\mathbf F_{I_3}^{-1}\otimes \mathbf I_{I_2})
=
\mathrm{bdiag}\!\bigl(
\widehat{\underline{\mathbf X}}(:,:,1),\dots,\widehat{\underline{\mathbf X}}(:,:,I_3)
\bigr),
\]
where $\mathrm{bdiag}(\cdot)$ denotes the block-diagonal matrix formed from the frontal slices in the Fourier domain.

As a consequence, if
\[
\underline{\mathbf C}=\underline{\mathbf X}*\underline{\mathbf Y},
\]
then the t-product reduces to independent matrix multiplications across Fourier slices:
\begin{equation}\label{eq:tprod-fourier}
\widehat{\underline{\mathbf C}}(:,:,t)
=
\widehat{\underline{\mathbf X}}(:,:,t)\,
\widehat{\underline{\mathbf Y}}(:,:,t),
\qquad t=1,\dots,I_3.
\end{equation}
The result is then recovered by inverse FFT along the third mode.

For real tensors, the Fourier slices satisfy the usual conjugate-symmetry relations:
\[
\widehat{\underline{\mathbf X}}(:,:,1)\in\mathbb{R}^{I_1\times I_2},
\qquad
\widehat{\underline{\mathbf X}}(:,:,i)
=
\mathrm{conj}\!\bigl(\widehat{\underline{\mathbf X}}(:,:,I_3-i+2)\bigr),
\]
for the appropriate range of indices. Therefore, only about half of the Fourier slices need to be processed explicitly in practice. This Fourier-domain procedure is summarized in Algorithm~\ref{ALG:TSVDP}.

\begin{algorithm}[ht!]
\caption{Fast t-product of two third-order tensors}\label{ALG:TSVDP}
\KwIn{Two tensors $\underline{\mathbf X}\in\mathbb{R}^{I_1\times I_2\times I_3}$ and $\underline{\mathbf Y}\in\mathbb{R}^{I_2\times I_4\times I_3}$.}
\KwOut{$\underline{\mathbf C}=\underline{\mathbf X}*\underline{\mathbf Y}\in\mathbb{R}^{I_1\times I_4\times I_3}$.}
$\widehat{\underline{\mathbf X}}\leftarrow \mathrm{fft}(\underline{\mathbf X},[],3)$\;
$\widehat{\underline{\mathbf Y}}\leftarrow \mathrm{fft}(\underline{\mathbf Y},[],3)$\;
\For{$i=1,\dots,\lceil (I_3+1)/2\rceil$}{
    $\widehat{\underline{\mathbf C}}(:,:,i)\leftarrow
    \widehat{\underline{\mathbf X}}(:,:,i)\,\widehat{\underline{\mathbf Y}}(:,:,i)$\;
}
\For{$i=\lceil (I_3+1)/2\rceil+1,\dots,I_3$}{
    $\widehat{\underline{\mathbf C}}(:,:,i)\leftarrow
    \mathrm{conj}\!\bigl(\widehat{\underline{\mathbf C}}(:,:,I_3-i+2)\bigr)$\;
}
$\underline{\mathbf C}\leftarrow \mathrm{ifft}(\widehat{\underline{\mathbf C}},[],3)$\;
\end{algorithm}

We next recall the matrix-like notions such as transpose, identity, orthogonality, and singular value decomposition that are induced by the t-product algebra and will be used repeatedly in the sequel.

\subsection{Transpose, orthogonality, and T-SVD}

The matrix-like notions used in the t-product algebra are defined so that they are consistent with \eqref{eq:tprod-bcirc} and \eqref{eq:tprod-fourier}.

\begin{defn}[Transpose]
Let $\underline{\mathbf X}\in\mathbb{R}^{I_1\times I_2\times I_3}$. Its transpose $\underline{\mathbf X}^T\in\mathbb{R}^{I_2\times I_1\times I_3}$ is obtained by transposing each frontal slice and reversing the order of slices $2,\dots,I_3$. In the Fourier domain, this corresponds to slicewise conjugate transpose.
\end{defn}

\begin{defn}[Identity tensor]
The identity tensor $\underline{\mathbf I}\in\mathbb{R}^{I\times I\times I_3}$ is the tensor whose first frontal slice is the identity matrix and whose remaining frontal slices are zero.
\end{defn}

\begin{defn}[Orthogonal tensor]
A tensor $\underline{\mathbf Q}\in\mathbb{R}^{I\times I\times I_3}$ is orthogonal if
\[
\underline{\mathbf Q}^T*\underline{\mathbf Q}
=
\underline{\mathbf Q}*\underline{\mathbf Q}^T
=
\underline{\mathbf I}.
\]
\end{defn}

These definitions lead to the tensor singular value decomposition.

\begin{defn}[T-SVD]
Let $\underline{\mathbf X}\in\mathbb{R}^{I_1\times I_2\times I_3}$. A tensor singular value decomposition of $\underline{\mathbf X}$ is
\[
\underline{\mathbf X}
=
\underline{\mathbf U}*\underline{\mathbf S}*\underline{\mathbf V}^{T},
\]
where $\underline{\mathbf U}\in\mathbb{R}^{I_1\times I_1\times I_3}$ and
$\underline{\mathbf V}\in\mathbb{R}^{I_2\times I_2\times I_3}$ are orthogonal, and
$\underline{\mathbf S}\in\mathbb{R}^{I_1\times I_2\times I_3}$ is f-diagonal, meaning that each frontal slice of $\widehat{\underline{\mathbf S}}$ is diagonal.
\end{defn}

The T-SVD is computed in practice by taking the matrix SVD of each frontal slice of $\widehat{\underline{\mathbf X}}$. Truncating these slice-wise SVDs yields a low-tubal-rank approximation, called the \emph{truncated T-SVD}. This operation is the main local building block in the proposed TTT-SVD construction.

A fixed-precision variant can likewise be defined: given a tolerance $\delta>0$, one may choose the smallest tubal rank $R$ such that the truncated T-SVD approximation satisfies
\[
\|\underline{\mathbf X}-\underline{\mathbf U}_R*\underline{\mathbf S}_R*\underline{\mathbf V}_R^T\|_F\le \delta.
\]
We will not introduce separate notation for this variant, but this viewpoint motivates the tolerance-based formulations considered later in the paper.

\begin{algorithm}[t]
\caption{Truncated T-SVD of a third-order tensor}\label{ALG:TQR}
\KwIn{A tensor $\underline{\mathbf X}\in\mathbb{R}^{I_1\times I_2\times I_3}$ and a target tubal rank $R$.}
\KwOut{A rank-$R$ truncated T-SVD $\underline{\mathbf X}\approx \underline{\mathbf U}_R*\underline{\mathbf S}_R*\underline{\mathbf V}_R^T$.}
$\widehat{\underline{\mathbf X}}\leftarrow \mathrm{fft}(\underline{\mathbf X},[],3)$\;
\For{$i=1,\dots,\lceil (I_3+1)/2\rceil$}{
    $[\widehat{\underline{\mathbf U}}(:,:,i),\widehat{\underline{\mathbf S}}(:,:,i),\widehat{\underline{\mathbf V}}(:,:,i)]
    \leftarrow \mathrm{svds}\bigl(\widehat{\underline{\mathbf X}}(:,:,i),R\bigr)$\;
}
\For{$i=\lceil (I_3+1)/2\rceil+1,\dots,I_3$}{
    $\widehat{\underline{\mathbf U}}(:,:,i)\leftarrow \mathrm{conj}\!\bigl(\widehat{\underline{\mathbf U}}(:,:,I_3-i+2)\bigr)$\;
    $\widehat{\underline{\mathbf S}}(:,:,i)\leftarrow \widehat{\underline{\mathbf S}}(:,:,I_3-i+2)$\;
    $\widehat{\underline{\mathbf V}}(:,:,i)\leftarrow \mathrm{conj}\!\bigl(\widehat{\underline{\mathbf V}}(:,:,I_3-i+2)\bigr)$\;
}
$\underline{\mathbf U}_R\leftarrow \mathrm{ifft}(\widehat{\underline{\mathbf U}},[],3)$\;
$\underline{\mathbf S}_R\leftarrow \mathrm{ifft}(\widehat{\underline{\mathbf S}},[],3)$\;
$\underline{\mathbf V}_R\leftarrow \mathrm{ifft}(\widehat{\underline{\mathbf V}},[],3)$\;
\end{algorithm}

\subsection{Hyper-vectors, hyper-matrices, and hyper-tensors}

The t-product is naturally defined for third-order tensors. To use it inside higher-order tensor networks while preserving a distinguished tube mode, it is convenient to reinterpret tensors as arrays whose entries are themselves tubes.

A matrix $\mathbf A\in\mathbb{R}^{I\times T}$ may be viewed as a \emph{hyper-vector} of length $I$, whose entries are tubes of length $T$. Likewise, a third-order tensor
\[
\underline{\mathbf X}\in\mathbb{R}^{I\times J\times T}
\]
may be viewed as a \emph{hyper-matrix} of size $I\times J$ with tube length $T$. More generally, an order-$(N+1)$ tensor
\[
\underline{\mathbf X}\in\mathbb{R}^{I_1\times I_2\times \cdots \times I_N\times T}
\]
is viewed as an $N$-way \emph{hyper-tensor} whose last mode is the tube mode.

This viewpoint is not only a change of notation. It allows us to keep the last mode as the convolutional mode governed by the t-product while organizing the remaining modes through a train network. That is precisely the mechanism behind the proposed tubal tensor train.

\begin{defn}[Tubal outer product]
Let $\underaccent{\tilde}{\mathbf a}_n$ be hyper-vectors of lengths $I_n$ and common tube length $T$, for $n=1,\dots,N$. Their tubal outer product is the hyper-tensor
\[
\underaccent{\tilde}{\mathbf Y}
=
\underaccent{\tilde}{\mathbf a}_1\circ
\underaccent{\tilde}{\mathbf a}_2\circ \cdots \circ
\underaccent{\tilde}{\mathbf a}_N
\]
defined entrywise by
\[
\underaccent{\tilde}{\mathbf Y}(i_1,\dots,i_N)
=
\underaccent{\tilde}{\mathbf a}_1(i_1)\circledast
\underaccent{\tilde}{\mathbf a}_2(i_2)\circledast \cdots \circledast
\underaccent{\tilde}{\mathbf a}_N(i_N).
\]
\end{defn}

\section{Tensor-train background and motivation for TTT}\label{Sec:tSVD}

The proposed model is motivated by two complementary observations. On the one hand, T-SVD provides a natural algebra for third-order tensors through the t-product. On the other hand, the tensor-train format is one of the most effective ways to represent high-order tensors with low-order cores. The tubal tensor train combines these two ideas.

\newpage
\subsection{Tensor-train decomposition}

Let
\[
\underline{\mathbf X}\in\mathbb{R}^{I_1\times I_2\times\cdots\times I_N}.
\]
A tensor-train (TT) decomposition represents each entry of $\underline{\mathbf X}$ as
\[
\underline{\mathbf X}(i_1,\ldots,i_N)
=
\sum_{r_1=1}^{R_1}\cdots\sum_{r_{N-1}=1}^{R_{N-1}}
\underline{\mathbf G}^{(1)}(i_1,r_1)\,
\underline{\mathbf G}^{(2)}(r_1,i_2,r_2)\cdots
\underline{\mathbf G}^{(N)}(r_{N-1},i_N),
\]
where the cores satisfy
\[
\underline{\mathbf G}^{(n)}\in\mathbb{R}^{R_{n-1}\times I_n\times R_n},
\qquad
R_0=R_N=1.
\]
The tuple $(R_1,\dots,R_{N-1})$ is called the TT-rank.

The main strength of TT is that an $N$th-order tensor is represented using only low-order cores. For bounded ranks and mode sizes, the storage grows only linearly with $N$, which makes TT highly effective for high-dimensional data.

\subsection{Why T-SVD alone is not enough for higher-order tensors}\label{t-svd_sec}

For third-order tensors, T-SVD is attractive because it preserves many familiar matrix-SVD properties while exploiting the tube-wise convolution structure. In particular, it provides a natural low-rank approximation mechanism through the truncated T-SVD; see Figure~\ref{t_svd}. However, when one extends T-SVD directly to higher-order tensors, the resulting factors retain the same order as the original data tensor. This limits the scalability of the approach and leads to a curse-of-dimensionality bottleneck.

This suggests the following design principle: keep the t-product algebra attached to the tube mode, but avoid high-order global factors by organizing the remaining modes through a train structure. This is precisely the idea of the tubal tensor train introduced in the next section, where the TTT structure is defined and illustrated.

\begin{figure}[ht!]
\centering
\includegraphics[width=0.7\textwidth]{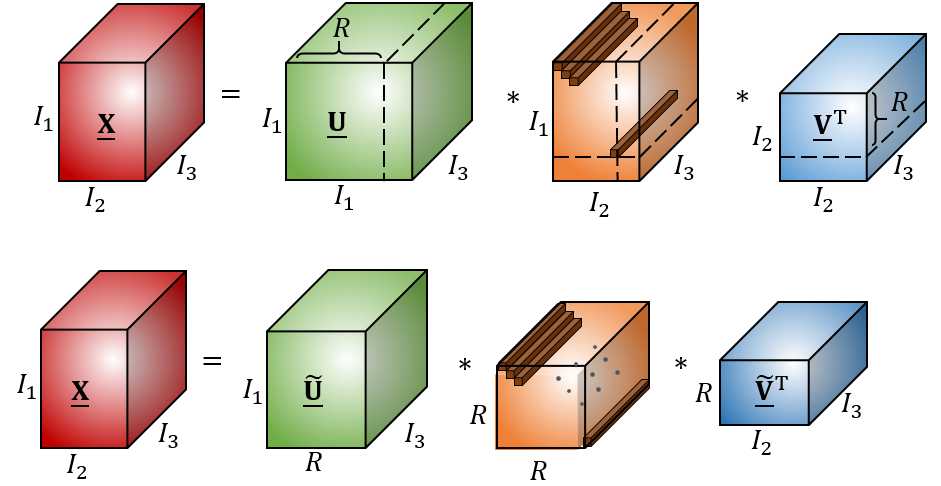}
\caption{T-SVD and truncated T-SVD for a third-order tensor.}
\label{t_svd}
\end{figure}

\section{Proposed Tubal Tensor-Train Decomposition}\label{Sec_TTT}

We now introduce the tubal tensor train (TTT) decomposition. The main idea is to represent a high-order tensor through a sequence of low-order tubal cores connected by the t-product. This keeps the convolutional structure induced by the tube mode while avoiding the high-order-core bottleneck of direct T-SVD extensions.

\subsection{Definition and interpretation}

Let $\underaccent{\tilde}{\mathbf X}$ be an $N$-way hyper-tensor of size $I_1\times I_2\times\cdots\times I_N$ with tube length $T$. A TTT representation with tubal ranks $(R_1,\dots,R_{N-1})$ consists of cores
\[
\underaccent{\tilde}{\mathbf X}^{(n)}\in\mathbb{R}^{R_{n-1}\times I_n\times R_n},
\qquad n=1,\dots,N,
\]
with $R_0=R_N=1$, such that
\begin{equation}\label{eq:ttt-sum}
\underaccent{\tilde}{\mathbf X}
=
\sum_{r_1=1}^{R_1}\cdots \sum_{r_{N-1}=1}^{R_{N-1}}
\underaccent{\tilde}{\mathbf X}^{(1)}(1,:,r_1)\circ
\underaccent{\tilde}{\mathbf X}^{(2)}(r_1,:,r_2)\circ \cdots \circ
\underaccent{\tilde}{\mathbf X}^{(N)}(r_{N-1},:,1).
\end{equation}
Equivalently, each tube of $\underaccent{\tilde}{\mathbf X}$ can be written as
\begin{equation}\label{eq:ttt-entry}
\underaccent{\tilde}{\mathbf X}(i_1,\dots,i_N)
=
\underaccent{\tilde}{\mathbf X}^{(1)}(1,i_1,:)
*
\underaccent{\tilde}{\mathbf X}^{(2)}(:,i_2,:)
*
\cdots
*
\underaccent{\tilde}{\mathbf X}^{(N)}(:,i_N,1).
\end{equation}

Thus, TTT may be viewed as a tubal matrix product state. The boundary cores are third-order tensors and the interior cores are fourth-order tensors when the explicit tube mode is displayed. More precisely, the total number of stored entries is
\[
R_1 I_1 T+\sum_{n=2}^{N-1} R_{n-1} I_n R_n\,T+R_{N-1} I_N T.
\]
In the uniform case $I_n=I$ and $R_n=R$, this becomes
\[
2IRT+(N-2)IR^2T=\mathcal{O}(NIR^2T).
\]
This linear scaling with the number of modes is one of the main advantages of TTT over direct higher-order extensions of T-SVD.

\begin{figure}[ht!]
\centering
\includegraphics[width=0.70\textwidth,trim=1cm 0 0 0,clip]{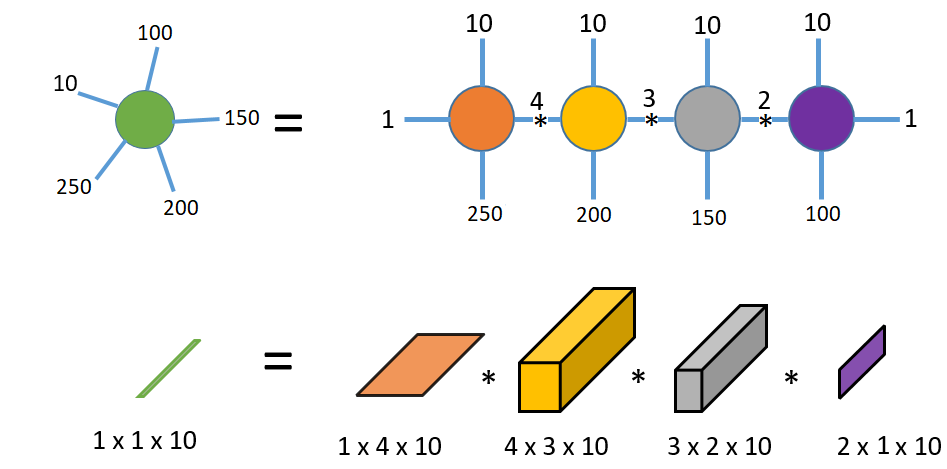}
\caption{Tubal tensor train structure and tube-wise evaluation.}
\label{TT_y}
\end{figure}

As an illustration, consider a tensor of size $100\times 150\times 200\times 250\times 10$, where the last mode is the tube mode. If the internal tubal rank profile is $(R_1,R_2,R_3)=(4,3,2)$ (equivalently, the full boundary-augmented profile is $(R_0,R_1,R_2,R_3,R_4)=(1,4,3,2,1)$), then the corresponding TTT cores have sizes
\[
1\times 100\times 4,\qquad
4\times 150\times 3,\qquad
3\times 200\times 2,\qquad
2\times 250\times 1,
\]
all with tube length $10$. Figure~\ref{TT_y} visualizes both the network structure and the tube-wise evaluation rule in \eqref{eq:ttt-entry}. We write
\[
\ll \underaccent{\tilde}{\mathbf Y}^{(1)},\dots,\underaccent{\tilde}{\mathbf Y}^{(N)}\gg
\]
for the TTT contraction of a sequence of tubal cores. Additional graphical illustrations of the TT structure, the truncated T-SVD, and the successive reshape--truncation steps underlying TTT-SVD are provided in Appendix~\ref{app:more_tt_ttt_algo}.

\subsection{A sequential fixed-rank construction: TTT-SVD}

For a prescribed tubal rank profile $(r_1,\dots,r_{N-1})$, the TTT decomposition can be computed by a sequential algorithm modeled after TT-SVD. The difference is that the unfolding step of TT-SVD is replaced by reshaping the current tensor into a third-order tensor, followed by a truncated T-SVD.

\begin{algorithm}[t]
\caption{TTT-SVD (fixed-rank version)}\label{Alg_tttr}
\KwIn{A hyper-tensor $\underaccent{\tilde}{\mathbf X}$ of size $I_1\times \cdots \times I_N$ with tube length $T$, and a tubal rank profile $(r_1,\dots,r_{N-1})$.}
\KwOut{Cores $\underaccent{\tilde}{\mathbf Y}^{(1)},\dots,\underaccent{\tilde}{\mathbf Y}^{(N)}$ such that $\underaccent{\tilde}{\mathbf X}\approx\ll \underaccent{\tilde}{\mathbf Y}^{(1)},\dots,\underaccent{\tilde}{\mathbf Y}^{(N)}\gg$.}
$\underaccent{\tilde}{\mathbf C}\leftarrow \mathrm{reshape}(\underaccent{\tilde}{\mathbf X},[I_1,\,I_2\cdots I_N,\,T])$\;
$[\underaccent{\tilde}{\mathbf U},\underaccent{\tilde}{\mathbf S},\underaccent{\tilde}{\mathbf V}]
\leftarrow \mathrm{truncated\_TSVD}(\underaccent{\tilde}{\mathbf C},r_1)$\;
$\underaccent{\tilde}{\mathbf Y}^{(1)}\leftarrow \underaccent{\tilde}{\mathbf U}$\;
$\underaccent{\tilde}{\mathbf C}\leftarrow \underaccent{\tilde}{\mathbf S}*\underaccent{\tilde}{\mathbf V}^{T}$\;
\For{$n=2,\dots,N-1$}{
    $\underaccent{\tilde}{\mathbf C}\leftarrow \mathrm{reshape}(\underaccent{\tilde}{\mathbf C},[r_{n-1}I_n,\,I_{n+1}\cdots I_N,\,T])$\;
    $[\underaccent{\tilde}{\mathbf U},\underaccent{\tilde}{\mathbf S},\underaccent{\tilde}{\mathbf V}]
    \leftarrow \mathrm{truncated\_TSVD}(\underaccent{\tilde}{\mathbf C},r_n)$\;
    $\underaccent{\tilde}{\mathbf Y}^{(n)}\leftarrow \mathrm{reshape}(\underaccent{\tilde}{\mathbf U},[r_{n-1},\,I_n,\,r_n,\,T])$\;
    $\underaccent{\tilde}{\mathbf C}\leftarrow \underaccent{\tilde}{\mathbf S}*\underaccent{\tilde}{\mathbf V}^{T}$\;
}
$\underaccent{\tilde}{\mathbf Y}^{(N)}\leftarrow \mathrm{reshape}(\underaccent{\tilde}{\mathbf C},[r_{N-1},\,I_N,\,T])$\;
\end{algorithm}

A fixed-precision variant is obtained by replacing the prescribed-rank truncations by local tolerances. Randomized T-SVD routines \cite{zhang2018randomized} or slice/CUR-type accelerations \cite{tarzanagh2018fast,ahmadi2024robust,ahmadi2023fastt,ahmadi2024adaptive} can also be incorporated. A schematic visualization of the sequential reshape--truncated-T-SVD procedure is given in Appendix~\ref{app:more_tt_ttt_algo}; see Figure~\ref{tsvd_demo}.

A practical limitation of TTT-SVD, as in TT-SVD, is that the sequential rank decisions may lead to unbalanced cores and unnecessarily large parameter counts for a target approximation error. This motivates the second algorithm below.

\subsection{A Fourier-domain alternating scheme: TATCU}

Consider the constrained approximation problem
\begin{equation}\label{EQ_1}
\min_{\underaccent{\tilde}{\mathbf Y}}\;
\|\underaccent{\tilde}{\mathbf X}-\underaccent{\tilde}{\mathbf Y}\|_F
\qquad
\text{subject to } \underaccent{\tilde}{\mathbf Y}
=
\ll \underaccent{\tilde}{\mathbf Y}^{(1)},\dots,\underaccent{\tilde}{\mathbf Y}^{(N)}\gg.
\end{equation}
In the TT setting, alternatives to TT-SVD such as ALS, DMRG, and ATCU often produce more balanced decompositions when the target is an error tolerance rather than a fixed rank profile \cite{holtz2012alternating,white1993density,khoromskij2010dmrg+,phan2020tensor}. Our TATCU strategy transfers this idea to the tubal setting through the Fourier domain.

The key observation is that the FFT along the tube mode turns the t-product into ordinary matrix multiplication on each frequency slice. Hence, after transforming $\underaccent{\tilde}{\mathbf X}$ along its tube mode, the TTT approximation problem decouples into a family of standard TT approximation problems, one for each Fourier slice. We therefore apply ATCU slice-wise in the Fourier domain, while distributing the prescribed global error tolerance across the frequency slices. The resulting slice-wise TT cores are then synchronized to a common spectral rank profile and assembled into tubal cores by inverse FFT.

\begin{algorithm}[t]
\caption{TATCU (Fourier-slice alternating two-cores update)}\label{Alg_tascu}
\KwIn{A hyper-tensor $\underaccent{\tilde}{\mathbf X}$ of size $I_1\times\cdots\times I_N$ with tube length $T$, and a target global relative tolerance $\varepsilon$.}
\KwOut{A TTT approximation $\underaccent{\tilde}{\mathbf Y}=\ll \underaccent{\tilde}{\mathbf Y}^{(1)},\dots,\underaccent{\tilde}{\mathbf Y}^{(N)}\gg$ obtained by slice-wise ATCU in the Fourier domain, targeting the global relative tolerance $\varepsilon$.}

$\widehat{\underline{\mathbf X}}\leftarrow \mathrm{fft}(\underaccent{\tilde}{\mathbf X},[],N+1)$\;
allocate slice-wise error budgets $\eta_k$ so that their total is consistent with the global tolerance $\varepsilon$ (for example, proportionally to the slice energies)\;
\For{$k=1,\dots,T$}{
    compute a TT approximation of the $k$th Fourier slice
    \[
    \widehat{\underline{\mathbf X}}(:,:,\dots,:,k)\approx
    \ll \mathbf C_k^{(1)},\mathbf C_k^{(2)},\dots,\mathbf C_k^{(N)}\gg
    \]
    using ATCU with local tolerance $\eta_k$\;
}
choose a common TT rank profile across frequencies, for example by taking the componentwise maximum and zero-padding smaller slice-wise cores\;
stack the padded spectral TT cores along the frequency index to form $\widehat{\underline{\mathbf Y}}^{(j)}$ for each $j$\;
$\underaccent{\tilde}{\mathbf Y}^{(1)}\leftarrow \mathrm{ifft}(\widehat{\underline{\mathbf Y}}^{(1)},[],3)$\;
\For{$j=2,\dots,N-1$}{
    $\underaccent{\tilde}{\mathbf Y}^{(j)}\leftarrow \mathrm{ifft}(\widehat{\underline{\mathbf Y}}^{(j)},[],4)$\;
}
$\underaccent{\tilde}{\mathbf Y}^{(N)}\leftarrow \mathrm{ifft}(\widehat{\underline{\mathbf Y}}^{(N)},[],3)$\;
verify the final global relative error and, if necessary, refine the slice-wise budgets or ranks and repeat the slice-wise approximation step\;
\end{algorithm}

\begin{rem}
The rank synchronization step in Algorithm~\ref{Alg_tascu} is essential. If ATCU is run independently on each Fourier slice, the resulting TT ranks may differ across frequencies. To assemble the slice-wise TT cores into tubal cores, one must therefore enforce a common spectral rank profile, for instance by taking the componentwise maximum and zero-padding the smaller slice-wise cores. Moreover, a prescribed global relative tolerance cannot in general be enforced by assigning the same local tolerance to every Fourier slice. In practice, one should distribute the global error budget across the slices, for example using Parseval's identity and the slice energies, and then verify the resulting global reconstruction error after inverse FFT. In this way, TATCU can be implemented as a globally tolerance-driven method while retaining its slice-wise ATCU structure.
\end{rem}

%
%

\subsection{Approximation error of TTT-SVD}

We now state the standard TT-SVD-type error bound for Algorithm~\ref{Alg_tttr}.

\begin{thm}\label{thm_1}
Consider Algorithm~\ref{Alg_tttr}. Suppose that at step $n$ the truncated T-SVD of the current hyper-matrix $\underaccent{\tilde}{\mathbf C}_n$ produces a local error bounded by
\[
\bigl\|\underaccent{\tilde}{\mathbf C}_n
-
\underaccent{\tilde}{\mathbf U}_n*
\underaccent{\tilde}{\mathbf S}_n*
\underaccent{\tilde}{\mathbf V}_n^{T}\bigr\|_F
\le \delta_n,
\qquad n=1,\dots,N-1.
\]
Then the final TTT approximation $\underaccent{\tilde}{\mathbf Y}$ returned by Algorithm~\ref{Alg_tttr} satisfies
\[
\|\underaccent{\tilde}{\mathbf X}-\underaccent{\tilde}{\mathbf Y}\|_F^2
\le
\sum_{n=1}^{N-1}\delta_n^2.
\]
\end{thm}

\begin{proof}
Apply the FFT along the tube mode. By Parseval's identity, the Frobenius norm is preserved up to a constant scaling, so it is enough to work in the Fourier domain. In that domain, each truncated T-SVD in Algorithm~\ref{Alg_tttr} becomes a collection of independent truncated matrix SVDs on the frontal slices. Moreover, the reshape operations used in Algorithm~\ref{Alg_tttr} only reorganize the non-tube modes and therefore commute with the slice-wise Fourier decomposition. Consequently, for each frequency slice, the entire TTT-SVD procedure is exactly the classical TT-SVD algorithm applied to the corresponding slice of $\widehat{\underaccent{\tilde}{\mathbf X}}$. The standard TT-SVD error analysis \cite{oseledets2011tensor} then yields, on each slice, a squared-error bound given by the sum of the squared local truncation errors. Summing these bounds over all frequencies and applying Parseval's identity once more gives the stated inequality in the original domain.
\end{proof}


As in the classical TT setting \cite{oseledets2011tensor,holtz2012alternating}, the admissible set of tensors whose TTT ranks do not exceed a prescribed profile is closed. Indeed, after applying the FFT along the tube mode, this class corresponds to a slice-wise family of tensors with bounded TT ranks, and the inverse FFT preserves closedness. Therefore, a best approximation exists. Let $\underaccent{\tilde}{\mathbf Y}^{\rm best}$ denote such a best approximation. By combining Theorem~\ref{thm_1} with the standard TT-SVD quasi-optimality argument applied slice-wise in the Fourier domain, one obtains
\[
\|\underaccent{\tilde}{\mathbf X}-\underaccent{\tilde}{\mathbf Y}\|_F
\le \sqrt{N-1}\,\|\underaccent{\tilde}{\mathbf X}-\underaccent{\tilde}{\mathbf Y}^{\rm best}\|_F.
\]

\section{Simulations}\label{Sec:Sim}
This section reports numerical experiments comparing the proposed model with several baseline algorithms.

All algorithms were implemented in MATLAB on a laptop computer with a 2.60~GHz Intel(R) Core(TM) i7-5600U processor and 8~GB of memory. The corresponding MATLAB codes are available at \url{https://github.com/TTTmodelICLR24/TTT_MatLab_ICLR24} and can be used to reproduce the experiments reported below. We also provide a Python implementation of TTT-SVD and TATCU at \url{https://github.com/vleplat/ttt_package}, together with lightweight demonstration scripts, including a synthetic chromatic gesture video compression example.

\begin{exa}\label{ex_1}
 ({\bf Color images}) In this experiment, we focus on color images as real-world data tensors. The peak signal-to-noise ratio (PSNR), structural similarity index measure (SSIM), and mean squared error (MSE) are used to assess reconstruction quality. They are defined by
\[
\mathrm{MSE}=\frac{\|\underline{\mathbf X}-\underline{\mathbf Y}\|_F^2}{\mathrm{numel}(\underline{\mathbf X})},
\qquad
\mathrm{PSNR}=10\log_{10}\!\left(\frac{255^2}{\mathrm{MSE}}\right),
\]
while SSIM is computed in the standard patch-based way,
\[
\mathrm{SSIM}=\frac{(2\mu_x\mu_y+C_1)(2\sigma_{xy}+C_2)}{(\mu_x^2+\mu_y^2+C_1)(\sigma_x^2+\sigma_y^2+C_2)}.
\]
We also report the relative error
\[
\frac{\|\underline{\mathbf X}-\underline{\mathbf Y}\|_F}{\|\underline{\mathbf X}\|_F}.
\]

We consider eight color images of size $512\times 512\times 3$ shown in Figure~\ref{sample_im}. We reshape the images to order-10 tensors of size $\underbrace{4\times\cdots\times 4}_{9}\times3$. We use the same relative error bound of $0.15$ for the TT-based model and the proposed TTT model, and compute the corresponding tensor decompositions. Table~\ref{Table_viimage} reports the PSNR, SSIM, and MSE values of the reconstructed images obtained by the two approaches. The reconstructed images themselves are displayed in Figure~\ref{resul_1}.

The proposed TTT model consistently improves the reconstruction quality of the benchmark images, as reflected by higher PSNR and SSIM values together with lower MSE. Visually, it better preserves both background content and structural details. More importantly, unlike direct higher-order extensions of T-SVD, the TTT representation uses only low-order cores, which avoids the high-order-core bottleneck.

 \begin{figure}[ht!]
\includegraphics[width=.95\linewidth]{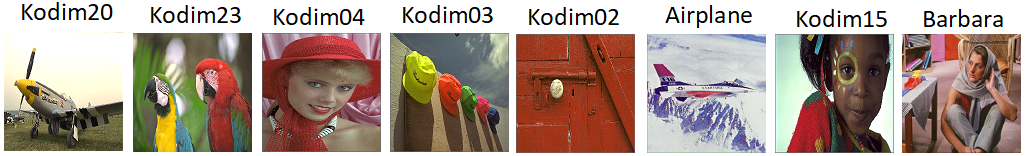}
\centering
\caption{\small Benchmark images used in our experiments.}
\label{sample_im}
\end{figure}

 \begin{table}
\begin{center}
\caption{\small The experimental results obtained by the proposed algorithm and the TT-based method for images and error bound of $0.15$.}\label{Table_viimage}
\vspace{0.2cm}
{\small
\begin{tabular}{c| r r r | r r r}
 \hline
  & \multicolumn{3}{c}{TT-Based} & \multicolumn{3}{c}{Proposed method}\\\hline
Images & MSE & PSNR (dB) & SSIM & MSE &  PSNR (dB) & SSIM\\
 \hline\hline
 {\rm Kodim03}& 251 & 24.12  & 0.6649 & {\bf 114.09} &  {\bf 27.52} & {\bf 0.7806} \\
 {\rm Kodim23}&  251.40 &  24.13 & 0.7042 & {\bf 127.27} & {\bf 27.08} & {\bf 0.8096}\\
  {\rm Kodim04}&  239.25 & 24.34 & 0.5940 & {\bf 93.46} & {\bf 28.42} & {\bf 0.7117}\\
    {\rm Airplane}& 874.30  & 18.71  & 0.6081 & {\bf 351.47} &{\bf 22.67} & {\bf 0.6921}\\
  {\rm Kodim15}& 380.32  & 22.33 & 0.6178 & {\bf 159.13} & {\bf 26.11} & {\bf 0.7673}\\
      {\rm Kodim20}& 436.51 &  21.73 & 0.6777 & {\bf 300.99} & {\bf 23.35} & {\bf 0.7171}\\
  {\rm Barbara}& 299.44 & 23.37  & 0.5746 & {\bf 121.21} & {\bf 27.30} & {\bf 0.7367}\\
  {\rm Kodim02}& 141.48 & 26.62 & 0.6692 & {\bf 66.56} & {\bf 29.90} & {\bf 0.7784}\\
 \hline
\end{tabular}
}
\end{center}
\end{table}

 \begin{figure}[ht!]
\includegraphics[width=.95\linewidth]{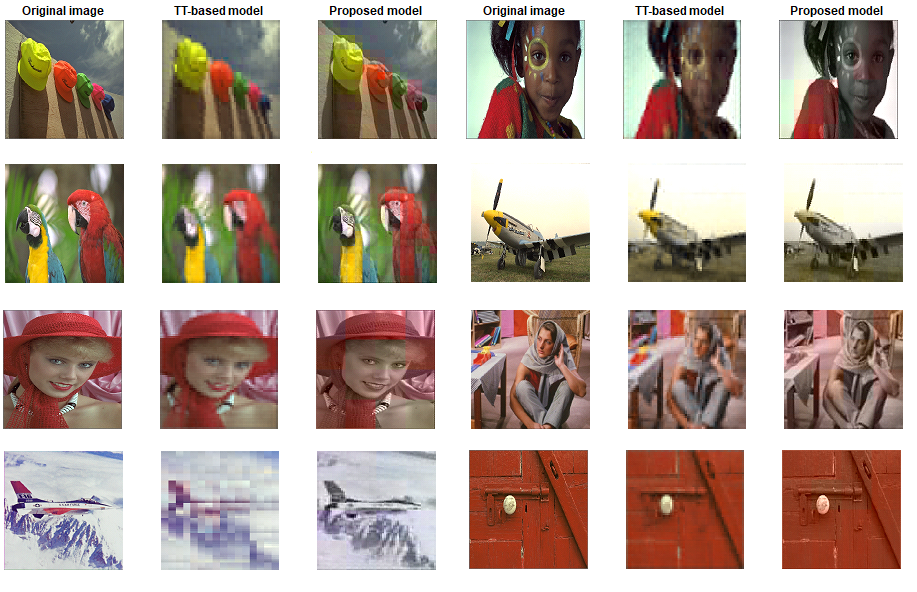}
\centering
\caption{\small The reconstructed images using the proposed TTT model and the TT-based model for an upper error bound of $0.15$.}
\label{resul_1}
\end{figure}

\end{exa}

\paragraph{Additional comparison with tensor chain.}
To strengthen the image-compression comparison, we also tested the tensor chain (TC) model \cite{espig2012note,zhao2016tensor}, which closes the TT network into a ring through an auxiliary index and is often more expressive than TT. Using the public implementation from \url{https://github.com/oscarmickelin/tensor-ring-decomposition}, we computed TTT, TT, and TC approximations for seven representative benchmark RGB images from Example~\ref{ex_1} with tolerance $0.15$. The resulting PSNR values are reported in Table~\ref{tab:comp}. As expected, TC improves over TT on most images, but it remains less competitive than the proposed TTT model.

\begin{table}[ht!]
\centering
\caption{PSNR values obtained by TTT, TT, and TC for tolerance $0.15$ on seven representative benchmark RGB images.}
\label{tab:comp}
\setlength{\tabcolsep}{4pt}
\resizebox{\textwidth}{!}{%
\begin{tabular}{c|ccccccc}
\hline
Method & Airplane & Barbara & Kodim02 & Kodim03 & Kodim04 & Kodim15 & Kodim20 \\
\hline
TTT & 22.67 & 27.30 & 29.90 & 27.52 & 28.42 & 26.11 & 23.35 \\
TT  & 18.71 & 23.37 & 26.62 & 24.12 & 24.34 & 22.33 & 21.73 \\
TC  & 20.85 & 26.01 & 27.78 & 26.34 & 26.75 & 25.04 & 21.17 \\
\hline
\end{tabular}%
}
\end{table}

\begin{exa}
 ({\bf Videos}) In this example, we consider the video data sets ``Akiyo'', ``News'', ``Tempete'', ``Waterfall'', ``Foreman'', and ``Stephan'' from \url{http://trace.eas.asu.edu/yuv/}. The proposed tensor model is compared with the TT decomposition and the T-SVD model in terms of compression factor and running time. Here, the compression factor is defined as the ratio between the number of entries of the original tensor and the number of stored parameters in the compressed representation. All videos have size $176\times 144\times 300$. We first reshape them into 10th-order tensors of size $4\times 4\times 9\times 4\times 4\times 11\times 3\times 4\times 5\times 5$. The experiments are divided into two parts. In the first part, we use the first three videos (``Akiyo'', ``News'', and ``Tempete'') and compare the proposed TTT model with the TT-based method under the same relative approximation error bound of $0.1$. The resulting compression factors and running times are reported in Table~\ref{Table_video}. The proposed method is faster on all three videos; moreover, for ``Akiyo\_qcif'' it yields a much larger compression factor, while for ``News\_qcif'' the compression factor is slightly lower than that of TT. The PSNR and SSIM values of all frames of the ``Akiyo\_qcif'' and ``News\_qcif'' videos are shown in Figure~\ref{video_1}. Overall, these results show that the proposed TTT model can provide comparable, and sometimes better, reconstruction quality for the same relative approximation error in less computing time.

 In the second part, we examine the three videos ``Waterfall'', ``Foreman'', and ``Stephan'' using the proposed TTT model and the T-SVD baseline with a relative error bound of $0.1$. The results are presented in Table~\ref{Table_video_2}. Compared with T-SVD, the proposed TTT model achieves significantly larger compression factors, albeit at a higher computational cost. These findings further support the relevance of the proposed model for video compression.

\begin{table}[ht!]
\centering
\caption{Experimental results for the proposed TTT algorithm and the TT-based baseline on the video data sets.}\label{Table_video}
\setlength{\tabcolsep}{4pt}
\resizebox{\textwidth}{!}{%
\begin{tabular}{c|ccc|ccc}
\hline
& \multicolumn{3}{c|}{TT-based model} & \multicolumn{3}{c}{Proposed method}\\
Video & Rel. err. & Time (s) & Comp. factor & Rel. err. & Time (s) & Comp. factor\\
\hline\hline
News\_qcif   & 0.1 & 23.80 & \textbf{13.34} & 0.1 & \textbf{21.73} & 12.44\\
Akiyo\_qcif  & 0.1 & 27.49 & 7.68           & 0.1 & \textbf{22.30} & \textbf{64.64}\\
Tempete\_cif & 0.1 & 32.16 & 2.3283         & 0.1 & \textbf{30.33} & \textbf{2.3284}\\
\hline
\end{tabular}%
}
\end{table}

\begin{table}[ht!]
\centering
\caption{Experimental results for the proposed TTT algorithm and the T-SVD baseline on the video data sets.}\label{Table_video_2}
\setlength{\tabcolsep}{4pt}
\resizebox{\textwidth}{!}{%
\begin{tabular}{c|ccc|ccc}
\hline
& \multicolumn{3}{c|}{T-SVD} & \multicolumn{3}{c}{Proposed method}\\
Video & Rel. err. & Time (s) & Comp. factor & Rel. err. & Time (s) & Comp. factor\\
\hline\hline
Waterfall\_cif & 0.1 & 10.34 & 5.04 & 0.1 & 20.73 & \textbf{20.10}\\
Foreman\_qcif  & 0.1 & 9.34  & 4.42 & 0.1 & 24.56 & \textbf{11.44}\\
Stephan\_qcif  & 0.1 & 11.28 & 2.50 & 0.1 & 24.90 & \textbf{2.53}\\
\hline
\end{tabular}%
}
\end{table}

\begin{figure}[ht!]
\includegraphics[width=0.48\linewidth]{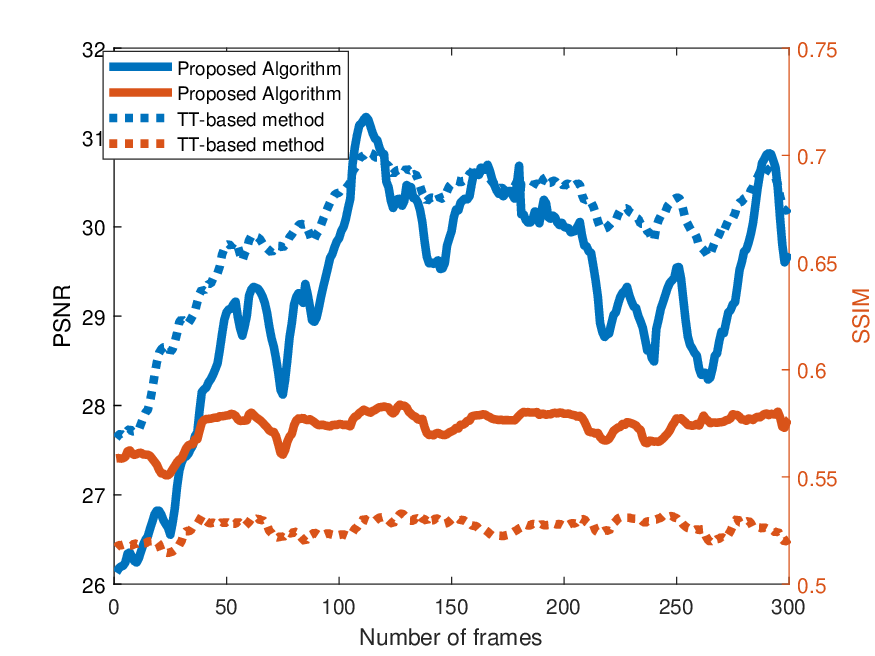}
\includegraphics[width=0.48\linewidth]{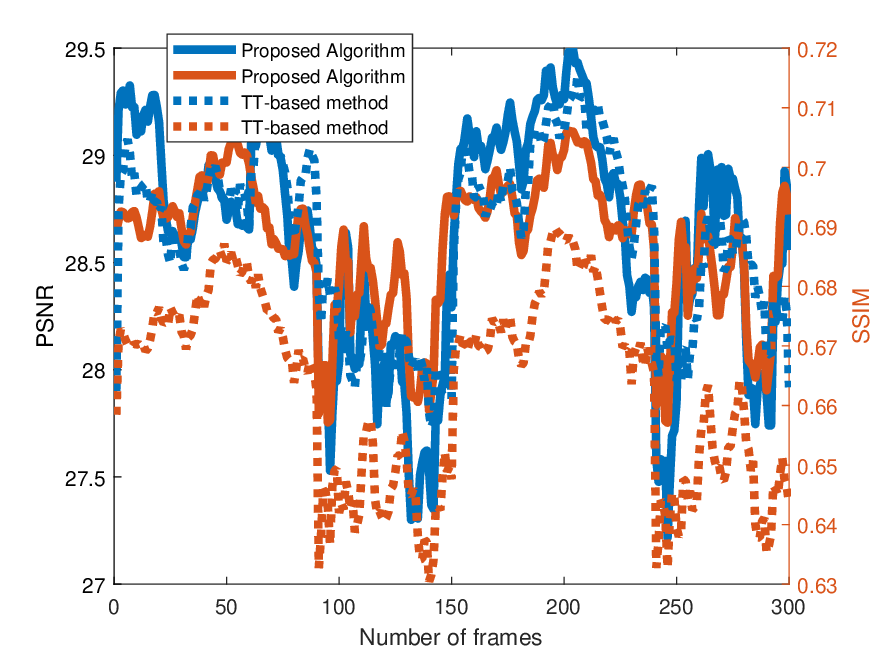}

\centering
\caption{\small The PSNR and SSIM of the reconstructed frames by the proposed TTT model and TT decomposition for the relative approximation error bound of $0.1$ for the ``Akiyo'' video (left) and the ``News'' video datasets.}
\label{video_1}
\end{figure}

\end{exa}

\begin{exa}\label{compl_exa}
({\bf Application to tensor completion}) We also examine the proposed model on a tensor completion task. Missing entries are common in practice, either because of acquisition failures or because only a subset of the data can be observed. Completion methods exploit low-rank structure to infer the missing values from the observed ones. We adopt the tensor completion scheme from \cite{ahmadi2023fast}, namely
\begin{equation}\label{Step1}
\underline{\mathbf X}^{(n)}\leftarrow \mathcal{L}(\underline{\mathbf C}^{(n)}),
\end{equation}
\begin{equation}\label{Step2}
\underline{\mathbf C}^{(n+1)}\leftarrow\underline{\mathbf \Omega}\oast\underline{\mathbf M}+(\underline{\mathbf 1}-\underline{\mathbf \Omega})\oast\underline{\mathbf X}^{(n)},
\end{equation}
where $\mathcal{L}$ computes a low-rank approximation of the current iterate $\underline{\mathbf C}^{(n)}$, $\underline{\mathbf 1}$ is the all-ones tensor, $\underline{\mathbf M}$ is the observed data tensor, and the indicator tensor $\underline{\mathbf \Omega}$ stores the positions of observed entries. In our experiments, $\mathcal{L}$ is instantiated either by a low-tubal-rank T-SVD approximation or by a low-TTT-rank approximation computed with the proposed model. We reshape the data tensor into an order-10 tensor of size $4\times 4\times 9\times 4\times 4\times 11\times 3\times 4\times 5\times 5$ and remove $70\%$ of its entries uniformly at random. The corresponding full boundary-augmented TTT rank profile is $[1,2,6,14,14,14,14,14,4,1]$ (equivalently, the internal ranks are $(2,6,14,14,14,14,14,4)$). Several tubal ranks are tested for T-SVD, and the best T-SVD result is reported. Figure~\ref{compl_resul} shows that TTT yields a clearly better reconstruction than T-SVD on this example.

 \begin{figure}[ht!]
\includegraphics[width=.95\linewidth]{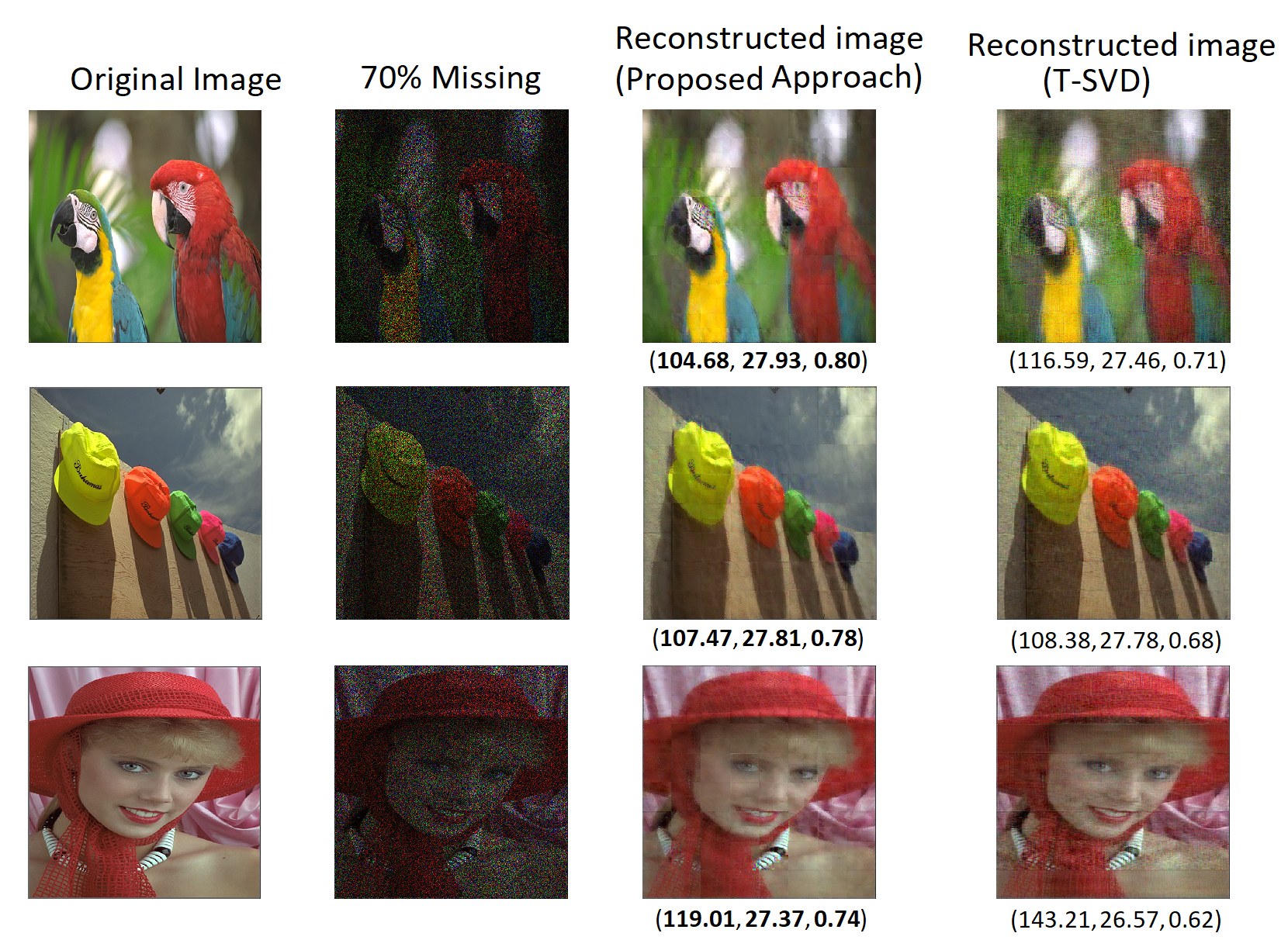}
\centering
\caption{\small Illustration of the reference data, the incomplete observations with $70\%$ of entries removed at random, and the reconstructions obtained with the proposed model and T-SVD for Example~\ref{compl_exa}. The displayed triplets correspond to ({\bf MSE, PSNR, SSIM}).}
\label{compl_resul}
\end{figure}

\end{exa}

\begin{exa}

({\bf Hyperspectral images}) We finally benchmark TTT and Algorithm~\ref{Alg_tascu} on hyperspectral images. Two comparison protocols are considered: fixed approximation accuracy (stage~1) and matched number of stored parameters (stage~2). The first protocol evaluates which model is more parsimonious at a prescribed reconstruction accuracy, whereas the second evaluates which model gives better reconstruction quality under a comparable storage budget. The quality measures are PSNR, RMSE, ERGAS, SAM, and UIQI. Table~\ref{tab:quali_measu2} reports representative results on one data set, while detailed descriptions of the data, the protocol, and the remaining numerical results are deferred to Appendices~\ref{subsec:dataset_desImage}--\ref{subsec:num_tests}.

\begin{center}
\begin{table*}[ht!]
\begin{center}
\caption{Comparison of tensor decomposition models on the ROSIS Pavia Univ. data set.
The table reports the quantitative quality detailed in Section \ref{subsec:PerfEval}.}
\label{tab:quali_measu2}
\begin{tabular}{|c|c|c|c|c|c|c|c|}
\hline
Method      & Runtime (sec.) & PSNR (dB) & RMSE & ERGAS & SAM & UIQI & \#Paras \\
\hline
Best value & 0  & $\infty$ & 0 & 0 & 0 & 1 & 0 \\
\hline
\multicolumn{8}{|c|}{Data set - ROSIS Pavia Univ. - $N=14$ - Relative error fixed and set to 0.08} \\
\hline

TTT & 61.26   & \textbf{35.89} & \textbf{130.02} & 10.77 & 3.01 & \textbf{0.97} & \textbf{12454744} \\
TT  & \textbf{42.93}   & 35.89 & 130.03 & \textbf{10.71}  & \textbf{2.82} & 0.97 & 13600746 \\

 \hline
 \multicolumn{8}{|c|}{Data set - ROSIS Pavia Univ. - $N=14$ - "equal" number of parameters } \\
\hline
TTT & 52.78   & \textbf{36.33}& \textbf{123.52} & \textbf{10.26} & 2.84 & \textbf{0.97} & \textbf{13565592}\\
TT  & \textbf{42.93}   & 35.89 & 130.03 & 10.71  & \textbf{2.82} & 0.97 & 13600746 \\

\hline
\end{tabular}
\end{center}
\end{table*}
\end{center}
\end{exa}

Across most hyperspectral benchmarks, the same qualitative trend is observed. At fixed accuracy (stage~1), TT is usually faster, whereas TTT tends to require fewer parameters while delivering comparable quality scores. At matched parameter count (stage~2), TTT generally provides the stronger reconstruction. Overall, the proposed decomposition and Algorithm~\ref{Alg_tascu} are competitive with strong TT baselines on this class of data.

\section{Conclusion}\label{sec:con}

We introduced the tubal tensor train (TTT) decomposition, a tensor-network format that combines the t-product structure of T-SVD with the low-order-core organization of TT. The resulting model replaces the high-order cores of direct T-SVD extensions by a train of third- and fourth-order tubal cores, thereby alleviating the curse of dimensionality while preserving the tube-wise convolutional structure.

Two computational strategies were presented: the sequential fixed-rank TTT-SVD construction and the Fourier-slice TATCU scheme. We also stated the standard TT-SVD-type error bound satisfied by TTT-SVD. The numerical experiments on images, videos, tensor completion, and hyperspectral data indicate that the proposed format is competitive and often advantageous in terms of reconstruction quality and compression.

Several directions remain open. On the algorithmic side, randomized and sketching-based variants of the local T-SVD steps are natural candidates for large-scale problems. On the modeling side, the same tubal-network principle can be combined with other tensor topologies such as tensor chain/ring or more general tensor networks. The TTT model was originally defined over the real numbers; its extension to complex numbers and quaternions is currently under investigation. These extensions will be investigated in future work.

\bibliographystyle{plainnat}
\bibliography{references}

\appendix

\section{Additional experiments on hyperspectral images}\label{sec:appen_HSI}
In this appendix, we consider hyperspectral images to evaluate the effectiveness of the proposed TTT model in comparison with the TT decomposition method. All the algorithms are implemented and tested on a laptop computer with Intel Core i7-11800H@2.30GHz CPU, and 16GB memory.
\subsection{Data sets}\label{subsec:dataset_desImage}
A hyperspectral image (HSI) is an image that contains information over a wide spectrum of light instead of just assigning primary colors (red, green, and blue) to each pixel as in RGB images. The spectral range of typical airborne sensors is 380-12700 nm and 400-1400 nm for satellite sensors. For instance, the AVIRIS airborne hyperspectral imaging sensor records spectral data over 224 continuous channels. The advantage of HSI is that they provide more information on what is imaged, some of it blind to the human eye as many wavelengths belong to the invisible light spectrum. This additional information allows one to identify and characterize the constitutive materials present in a scenery. We consider the following real HSIs: \begin{itemize}
    \item HYDICE Urban:
    The Urban data set\footnote{\url{http://lesun.weebly.com/hyperspectral-data-set.html} \label{note1} } consists of 307$\times$307 pixels and 162 spectral reflectance bands in the wavelength range 400nm to 2500nm.

    \item ROSIS Pavia University: The Pavia University data set\textsuperscript{\ref{note1}} was acquired by the ROSIS sensor during a flight campaign over Pavia, northern Italy, and consists of 610$\times$340 pixels and 109 spectral reflectance bands.
    \item AVIRIS Moffett Field: this data set has been acquired with over Moffett Field (CA, USA) in 1997 by the JPL spectro-imager AVIRIS \footnote{\url{https://aviris.jpl.nasa.gov/data/image_cube.html} } and consists of 512$\times$614 pixels and 224 spectral reflectance bands in the wavelength range 400nm to 2500nm. Due to the water vapor and atmospheric effects, we remove the noisy spectral bands. After this process, there remains 159 bands.

    \item AVIRIS Kennedy Space Center: this data set\footnote{\url{http://www.ehu.eus/ccwintco/index.php?title=Hyperspectral_Remote_Sensing_Scenes} } has been acquired with NASA AVIRIS instrument over the Kennedy Space Center (KSC), Florida, on March 23, 1996. AVIRIS acquires data in 224 bands of 10nm width with center wavelengths from 400nm to 2500nm. After removing water absorption and low SNR bands, 176 bands are used for the analysis. We finally extract a 250$\times$250 subimage from this dataset for our numerical experiments.
\end{itemize}

\subsection{Test procedure}\label{sebsec:HSITestProc}

We conduct a comprehensive benchmark of our proposed TTT model against the TT-based decomposition model for each dataset outlined in \ref{subsec:dataset_desImage}. Note that the data sets are first folded into a $N$-order tensor before computing the tensor decompositions. The testing procedure consists of two successive and complementary stages:

\begin{enumerate}
    \item In the first stage, we impose the same approximation error bound for both models and evaluate the performance of the approximations, as well as the number of parameters for each model. It is important to note that, for this setup, we utilize Algorithm \ref{Alg_tttr} to compute the TTT decomposition.

    \item In the second stage, we adapt the approximation error bound such that each model has an equal number of parameters. We then assess the quality of both decompositions under this constraint.
\end{enumerate}

The overall objective is to provide a fair and comprehensive comparison of the two models.

\subsubsection{Performance evaluation}\label{subsec:PerfEval}

In order to evaluate the quality of the tensor decompositions, we consider five widely used and complementary quality measurements:

\begin{itemize}
    \item Peak Signal-to-Noise Ratio (PSNR): PSNR assesses the spatial reconstruction quality of each band. It measures the ratio between the maximum power of a signal and the power of residual errors. A higher PSNR value indicates better spatial reconstruction quality.

    \item Root Mean Square Error (RMSE): RMSE is a similarity measure between the input tensor/image $\underline{\bf X}$ and the approximated tensor/image $\widehat{\underline{\bf X}}$. A smaller RMSE value indicates better fusion quality.

    \item Erreur Relative Globale Adimensionnelle de Synthèse (ERGAS): ERGAS provides a macroscopic statistical measure of the quality of the fused data. It calculates the amount of spectral distortion in the image \cite{Wald2000_mes}. The best value for ERGAS is 0.

    \item Spectral Angle Mapper (SAM): SAM quantifies the preservation of spectral information at each pixel. It computes the angle between two vectors of the estimated and reference spectra, resulting in a spectral distance measure. The overall SAM is obtained by averaging the SAMs computed for all image pixels. A smaller absolute value of SAM indicates better quality for the approximated tensor $\widehat{\underline{\bf X}}$.

    \item Universal Image Quality Index (UIQI): UIQI evaluates the similarity between two single-band images. It considers correlation, luminance distortion, and contrast distortion of the estimated image with respect to the reference image \cite{Wang2002}. The UIQI indicator ranges from -1 to 1. For multiband images, the overall UIQI is computed by averaging the UIQI values computed band by band. The best value for UIQI is 1.
\end{itemize}

For more details about these quality measurements, we refer the reader to \cite{Loncan2015}  and \cite{wei2016}.

\subsection{Experimental results}\label{subsec:num_tests}
The performance of each algorithm is presented in Table \ref{tab:quali_measu2} and Tables \ref{tab:quali_measu1} to \ref{tab:quali_measu4}. It is important to note that, for the K.S.C. dataset, despite our best efforts, we were unable to achieve an equal number of parameters for both models. The TT model consistently maintained a stable number of parameters around 3e6, which was significantly lower than the number of parameters obtained for the TTT model, even at higher levels of accuracy. As a result, we modified the test procedure for this dataset by using a lower level of accuracy in stage 2. In order to give more insights on the performance comparison between models,  Figures~\ref{fig:SamMaps_Urban} to \ref{fig:SamMaps_Moffett} display the SAM maps obtained for three first HSI data sets detailed in Section \ref{sec:appen_HSI} for both stages of the test procedure.

For the first three datasets, we observe a consistent trend: at the same level of accuracy (stage 1), the TT model demonstrates faster computation time compared to the TTT model. However, the TTT model exhibits a lower number of parameters, and both models yield similar values for the various quality measurements. Moving to stage 2, the TTT model consistently outperforms the TT model overall. Analyzing the SAM maps in Figures~\ref{fig:SamMaps_Urban} to \ref{fig:SamMaps_Moffett}, we observe that the TTT model achieves higher accuracy in areas with higher gradients. On the other hand, the TT model shows slightly better accuracy in less variable areas such as those corresponding to trees and grass (see, for instance, the right-hand side of the SAM maps for the Urban data set). Notably, in Figure~\ref{fig:SamMaps_Moffett}, during stage 2, the TTT model significantly outperforms the TT model in accurately reconstructing the water (northern part of the image).

In the case of the last dataset, we observe that the TTT model requires a significantly higher number of parameters to achieve both accuracy levels. However, it is worth noting that despite this difference, the TTT model demonstrates faster computation time and yields better values for the various quality measurements.

\begin{center}
\begin{table*}[ht!]
\begin{center}
\caption{Comparison of tensor decomposition models on the HYDICE Urban data set.
The table reports the quantitative quality detailed in Section \ref{subsec:PerfEval}.}
\label{tab:quali_measu1}
\begin{tabular}{|c|c|c|c|c|c|c|c|}
\hline
Method      & Runtime (sec.) & PSNR (dB) & RMSE & ERGAS & SAM & UIQI & \#Paras \\
\hline
Best value & 0  & $\infty$ & 0 & 0 & 0 & 1 & 0 \\
\hline
\multicolumn{8}{|c|}{Data set - HYDICE Urban - $N=13$ - Relative error fixed and set to 0.08} \\
\hline

TTT & 39.12   & 33.67 & \textbf{16.41} & 10.31 & 3.18 & 0.97 & \textbf{9010360} \\
TT  & \textbf{22.04}   & \textbf{33.80} & 16.42 & \textbf{10.08}  & \textbf{3.02} & 0.98 & 10670310 \\

 \hline
 \multicolumn{8}{|c|}{Data set - HYDICE Urban - $N=13$ - "equal" number of parameters } \\
\hline
TTT & 23.81   & \textbf{35.92}& \textbf{12.71} & \textbf{7.95} & \textbf{2.60} & \textbf{0.98} & \textbf{10559780}\\
TT  & \textbf{22.04}   & 33.80 & 16.42 & 10.08  & 3.02 & 0.98 & 10670310 \\

\hline
\end{tabular}
\end{center}
\end{table*}
\end{center}

\begin{center}
\begin{table*}[ht!]
\begin{center}
\caption{Comparison of tensor decomposition models on the AVIRIS Moffett data set.
The table reports the quantitative quality detailed in Section \ref{subsec:PerfEval}.}
\label{tab:quali_measu3}
\begin{tabular}{|c|c|c|c|c|c|c|c|}
\hline
Method      & Runtime (sec.) & PSNR (dB) & RMSE & ERGAS & SAM & UIQI & \#Paras \\
\hline
Best value & 0  & $\infty$ & 0 & 0 & 0 & 1 & 0 \\
\hline
\multicolumn{8}{|c|}{Data set - Moffett - $N=10$ - Relative error fixed and set to 0.02} \\
\hline

TTT & 0.68            & 39.72 & 0.01 & 4.90 & 6.22 & 1.00 & \textbf{362150} \\
TT  &  \textbf{0.28}  & \textbf{40.61} & 0.01 & \textbf{3.10}  & \textbf{5.78} & 1.00 & 418905 \\

 \hline
 \multicolumn{8}{|c|}{Data set - Moffett - $N=10$ - "equal" number of parameters } \\
\hline
TTT & \textbf{0.66}   & \textbf{44.15} & \textbf{0.00} & \textbf{2.93} & \textbf{4.06} & 1.00 & \textbf{407975} \\
TT  &  \textbf{0.28}  & 40.61 & 0.01 & 3.10  & 5.78 & 1.00 & 418905 \\

\hline
\end{tabular}
\end{center}
\end{table*}
\end{center}

\begin{center}
\begin{table*}[ht!]
\begin{center}
\caption{Comparison of tensor decomposition models on the AVIRIS Kennedy Space Center data set.
The table reports the quantitative quality detailed in Section \ref{subsec:PerfEval}.}
\label{tab:quali_measu4}
\begin{tabular}{|c|c|c|c|c|c|c|c|}
\hline
Method      & Runtime (sec.) & PSNR (dB) & RMSE & ERGAS & SAM & UIQI & \#Paras \\
\hline
Best value & 0  & $\infty$ & 0 & 0 & 0 & 1 & 0 \\
\hline
\multicolumn{8}{|c|}{Data set - K.S.C. - $N=14$ - Relative error fixed and set to 0.05} \\
\hline

TTT & \textbf{19.39}   & \textbf{39.99} & \textbf{130.13} & \textbf{86.88} & \textbf{25.16} & \textbf{0.49} & 8534545 \\
TT  &  26.90  & 38.28 & 130.35 & 119.49  & 37.73 & 0.32 & \textbf{3484834} \\

 \hline
 \multicolumn{8}{|c|}{Data set - K.S.C. - $N=14$ - Relative error fixed and set to 0.02 } \\
\hline
TTT & \textbf{19.28}   & \textbf{45.83} & \textbf{51.52} & \textbf{47.42} & \textbf{16.73} & \textbf{0.58} & 9188870 \\
TT  &  25.52  & 44.14 & 52.14 & 55.26  & 25.72 & 0.42 & \textbf{3976714} \\

\hline
\end{tabular}
\end{center}
\end{table*}
\end{center}

\begin{figure}[ht!]
	\centering
	\begin{subfigure}{0.49\linewidth}
		\includegraphics[width=\linewidth]{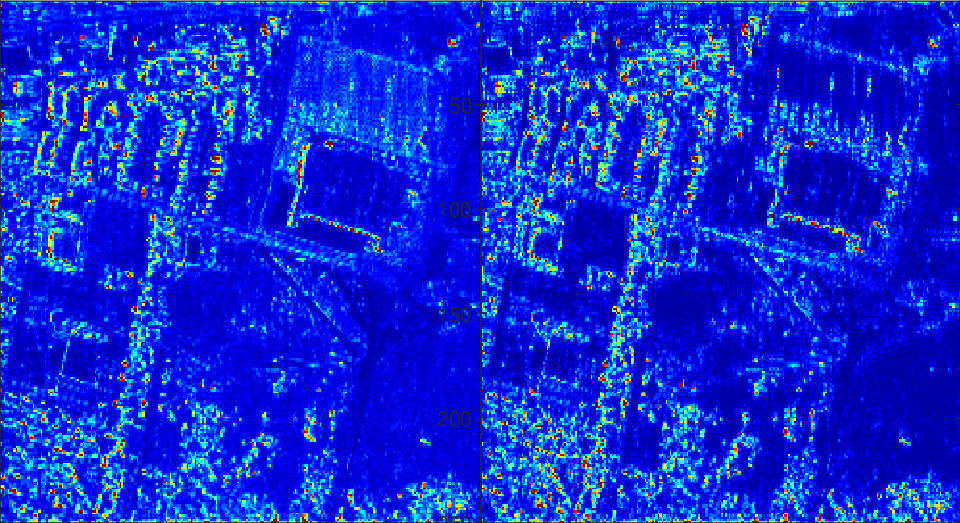}
		\caption{Stage 1: TTT (left) and TT (right)}
            \label{fig:Urban_stage1}
	\end{subfigure}
	\begin{subfigure}{0.49\linewidth}
		\includegraphics[width=\linewidth]{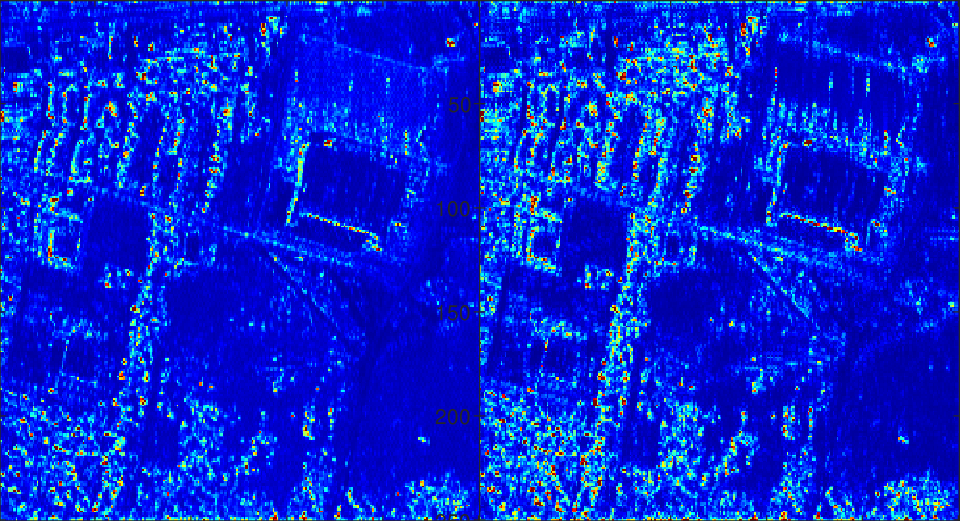}
		\caption{Stage 2: TTT (left) and TT (right)}
            \label{fig:Urban_stage2}
	\end{subfigure}
	\caption{SAM maps for HYDICE Urban data set for both stages of the test procedure detailed in Section \ref{sebsec:HSITestProc}, the values for sam quality measurement are limited to the interval [0;20].}
        \label{fig:SamMaps_Urban}
\end{figure}

\begin{figure}[ht!]
	\centering
	\begin{subfigure}{0.49\linewidth}
		\includegraphics[width=\linewidth]{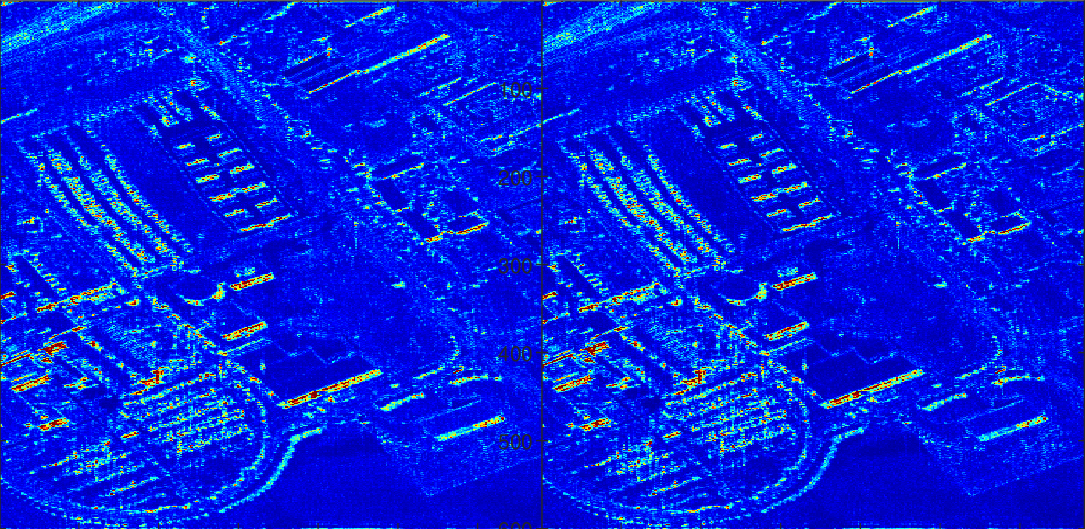}
		\caption{Stage 1: TTT (left) and TT (right)}
            \label{fig:PAvia_stage1}
	\end{subfigure}
	\begin{subfigure}{0.49\linewidth}
		\includegraphics[width=\linewidth]{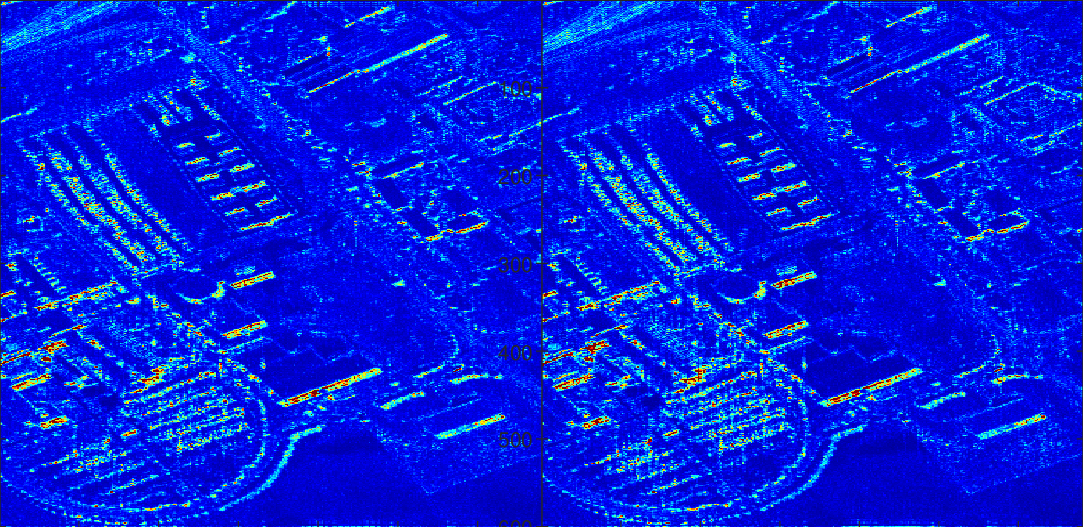}
		\caption{Stage 2: TTT (left) and TT (right)}
            \label{fig:PAvia_stage2}
	\end{subfigure}
	\caption{SAM maps for ROSIS Pavia University data set for both stages of the test procedure detailed in Section \ref{sebsec:HSITestProc}, the values for sam quality measurement are limited to the interval [0;20].}
        \label{fig:SamMaps_Pavia}
\end{figure}

\begin{figure}[ht!]
	\centering
	\begin{subfigure}{0.49\linewidth}
		\includegraphics[width=\linewidth]{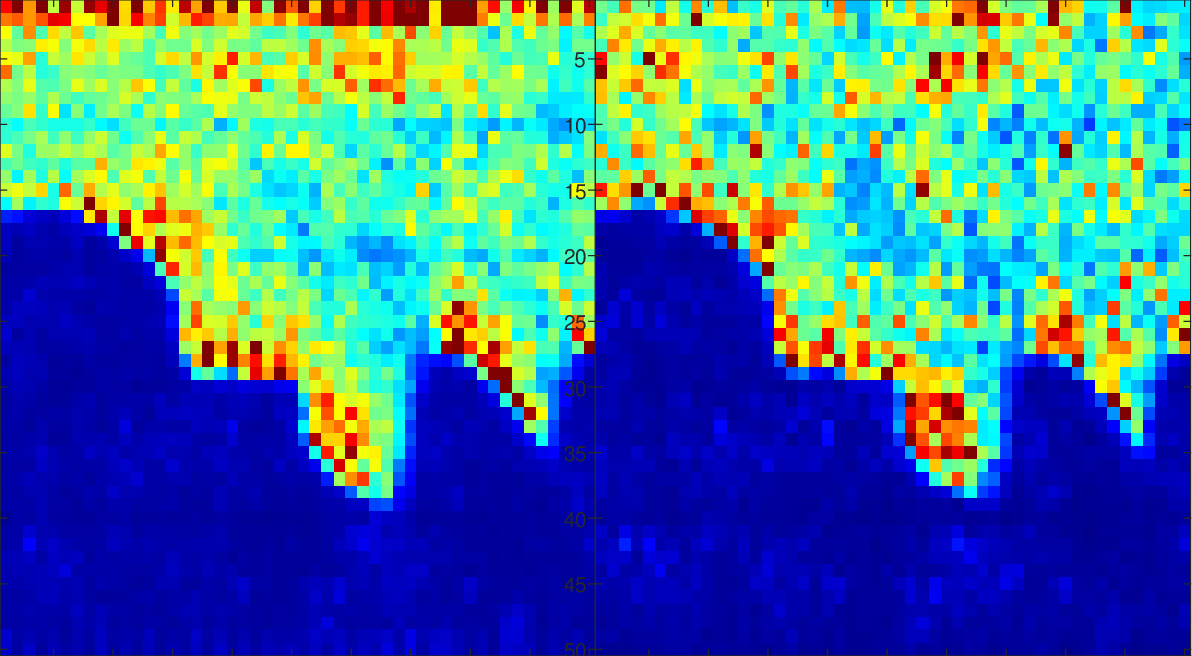}
		\caption{Stage 1: TTT (left) and TT (right)}
            \label{fig:Moffett_stage1}
	\end{subfigure}
	\begin{subfigure}{0.49\linewidth}
		\includegraphics[width=\linewidth]{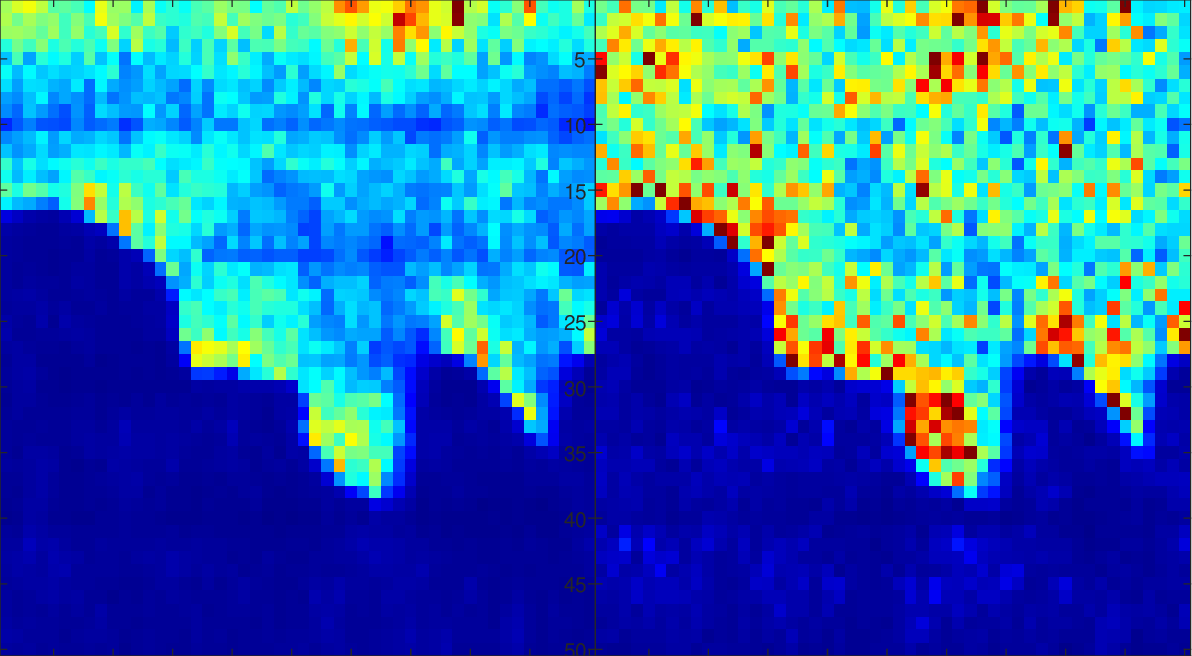}
		\caption{Stage 2: TTT (left) and TT (right)}
            \label{fig:Moffett_stage2}
	\end{subfigure}
	\caption{SAM maps for AVIRIS Moffett data set for both stages of the test procedure detailed in Section \ref{sebsec:HSITestProc}, the values for sam quality measurement are limited to the interval [0;20].}
        \label{fig:SamMaps_Moffett}
\end{figure}

\section{Supplementary graphical illustrations of TT, T-SVD, and TTT-SVD}\label{app:more_tt_ttt_algo}
This appendix provides graphical illustrations complementing the definitions and algorithms in the main text. Figure~\ref{TT} recalls the TT decomposition as a contraction of a sequence of low-order cores: an $N$th-order tensor is represented by two boundary matrices and $(N-2)$ third-order interior cores. The T-SVD was introduced in Section~\ref{Sec:prelim}, and its role as a motivation for TTT was discussed in Section~\ref{t-svd_sec}. Figures~\ref{t_svd_2} and \ref{t_svd_3} visualize the truncated T-SVD as a sum of tubal rank-1 terms and illustrate the corresponding viewpoint for third- and fourth-order tensors. Finally, Figure~\ref{tsvd_demo} provides a schematic description of the successive reshape--truncated-T-SVD steps used by Algorithm~\ref{Alg_tttr} to construct the TTT decomposition.

\begin{figure}[ht!]
\includegraphics[width=.7\linewidth]{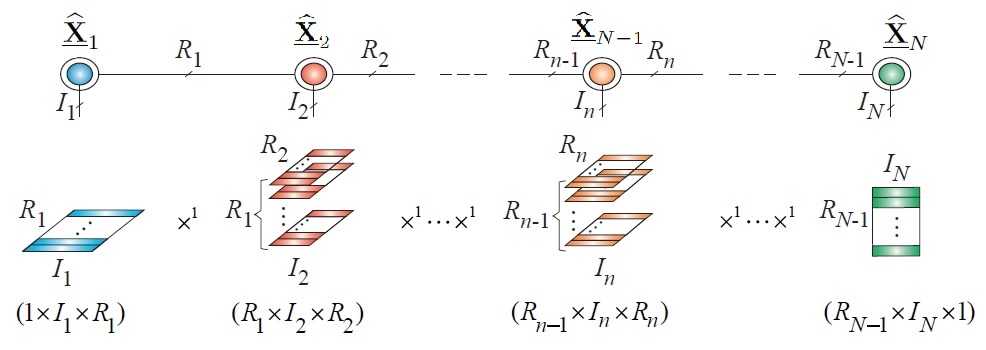}
\centering
\caption{\small Structure of the TT decomposition for an $N$-th order tensor.}
\label{TT}
\end{figure}

\begin{figure}[ht!]
\includegraphics[width=.5\linewidth]{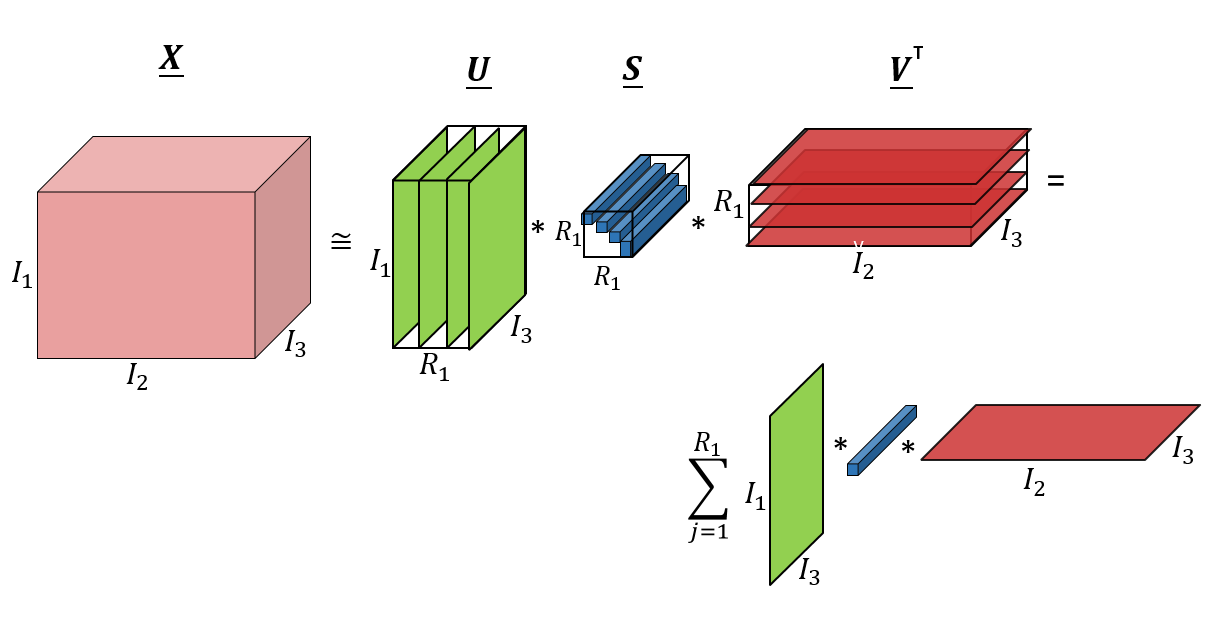}
\centering
\caption{\small The truncated T-SVD of a tensor $\underline{\bf X}$ as a sum of tubal rank-1 terms.}
\label{t_svd_2}
\end{figure}

\begin{figure}[ht!]
\includegraphics[width=.5\linewidth]{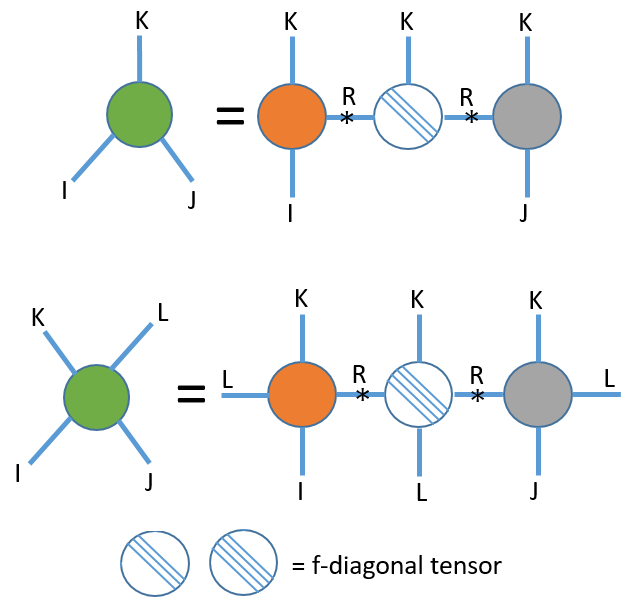}
\centering
\caption{\small The truncated T-SVD of a third-order and a fourth-order tensor.}
\label{t_svd_3}
\end{figure}

\begin{figure}[ht!]
\includegraphics[width=0.5\linewidth]{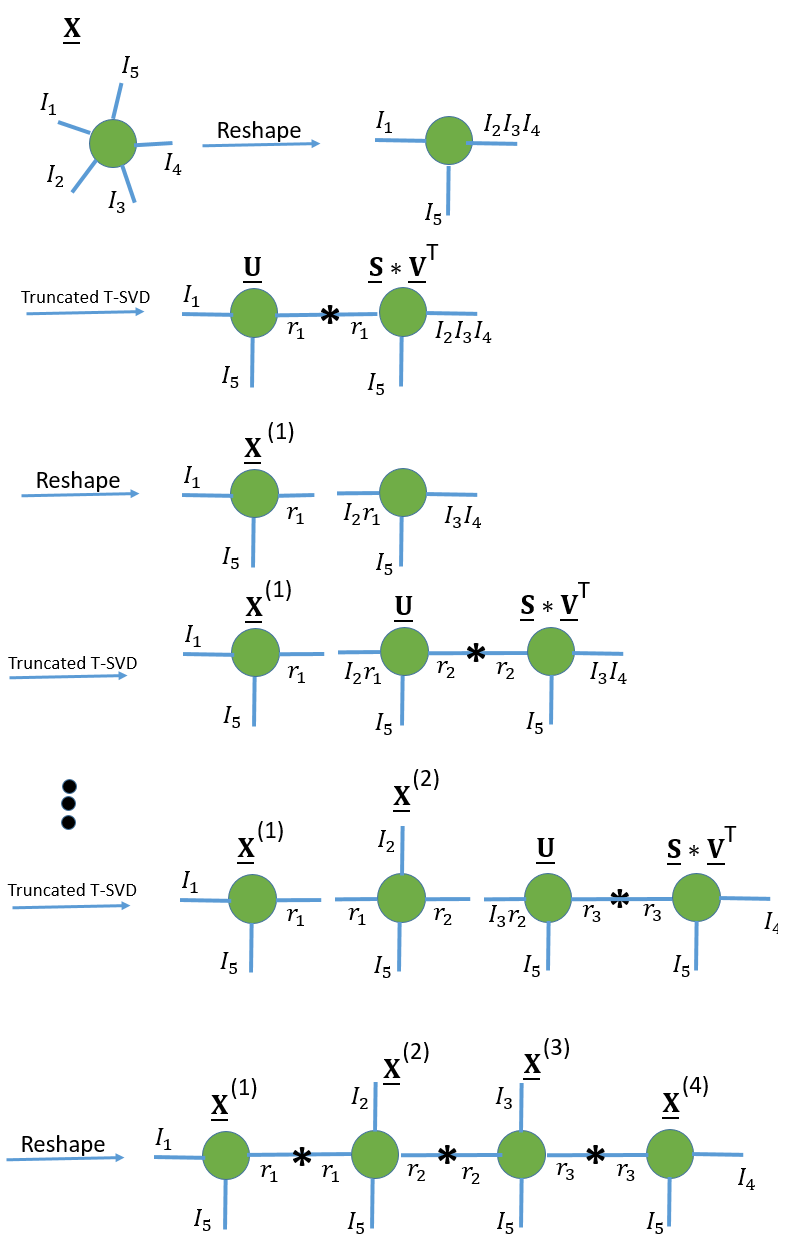}
\centering
\caption{\small The procedure of the TTT-SVD for decomposing a tensor into the TTT format.}
\label{tsvd_demo}
\end{figure}
\subsection{Other generalized tubal tensor models}
The T-SVD can also be combined with other tensor-network structures such as tensor ring/tensor chain (TR/TC), CPD, tensor wheel decomposition \cite{wu2022tensor}, or more general networks. Figure~\ref{TTR} gives a graphical illustration of a tubal TC/TR (TTC/TTR) construction. In that setting, an $N$th-order tensor is represented by $N$ fourth-order tubal cores arranged in a cycle, so that a tube of the original tensor is obtained through the trace of a cyclic sequence of t-products; for a third-order tensor $\cX$, this trace is defined by ${\rm Trace}(\cX)=\sum_{i=1}^{I_1}\cX(i,i,:)$. In analogy with the classical TR/TC model, the TTC/TTR format is characterized by a cyclic rank profile $(r_0,r_1,\ldots,r_{N-1})$, where the indices are understood modulo $N$ (equivalently, one may write $r_N:=r_0$ to close the cycle). The TTT model is recovered as the open-boundary analogue obtained when the cyclic connection is removed and the two boundary ranks are fixed to one. For the tubal HOSVD, see \cite{wang2022hot}. A detailed study of these generalized tubal models is beyond the scope of the present paper and is left for future work. Moreover, the proposed algorithms could in principle be accelerated using randomization, which will also be investigated in future work.

\begin{figure}[ht!]
\includegraphics[width=0.6\linewidth]{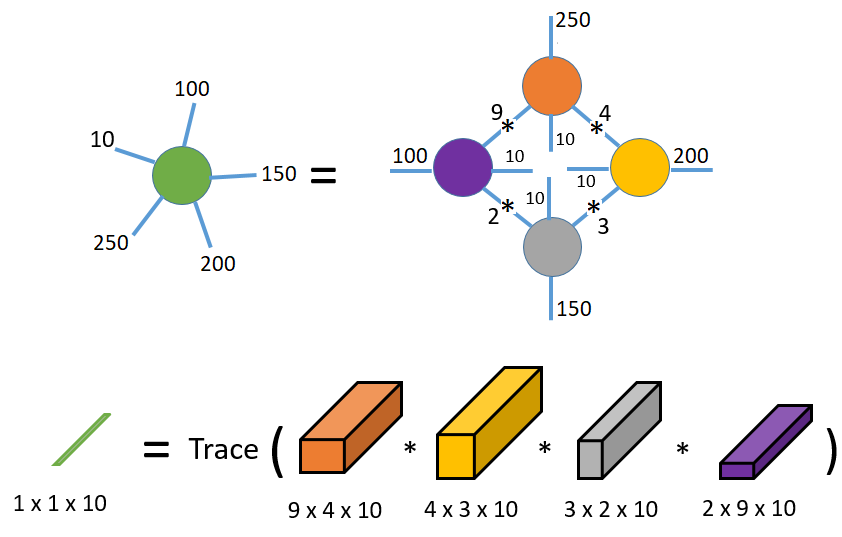}
\centering
\caption{\small (Upper) The visualization of the structure of the proposed tubal TC/TR decomposition. The connection between the core tensors is the T-product. (Bottom) A tube of the original tensor is obtained through a sequence of t-products involving sub-tensors of the TTC/TTR cores.}
\label{TTR}
\end{figure}
\end{document}